\documentclass[a4paper]{article}

\usepackage[utf8]{inputenc}
\usepackage{amsmath, amscd}
\usepackage{amsthm}
\usepackage{amssymb}
\usepackage[T1]{fontenc}
\usepackage[french, main=english]{babel}

\usepackage{microtype}

\usepackage{hyphenat} 

\usepackage{tikz}
\usetikzlibrary{arrows.meta, cd, decorations.pathmorphing}

\usepackage{multirow}

\usepackage{verbatim} %

\usepackage{pdfpages} %
\usepackage{rotating}

\usepackage{todonotes}

\usepackage{hyperref} %
\hypersetup{hidelinks}

\numberwithin{equation}{section}

\usepackage{enumerate}

\usepackage{faktor}

\usepackage[pagestyles]{titlesec}
\newpagestyle{main}{
  \sethead{\thesection}{}{\sectiontitle}
  \setfoot{}{\thepage}{}
  \headrule}

\DeclareMathOperator\botimes{\bar{\otimes}}

\newcommand{\Z}{\mathbb{Z}}
\newcommand{\N}{\mathbb{N}}
\newcommand{\F}{\mathbb{F}}
\newcommand{\Q}{\mathbb{Q}}

\newcommand{\M}{\mathcal{M}}
\newcommand{\THH}{\operatorname{THH}}

\newtheorem{proposition}[equation]{Proposition}
\newtheorem{lemma}[equation]{Lemma}

\newtheorem{theorem}[equation]{Theorem}
\newtheorem{definition}[equation]{Definition}
\newtheorem{axiom}[equation]{Axiom}

\title{A computation of $\THH_*(ku)$ using a gathered spectral sequence}
\author{Maxime \sc{Chaminadour}}
\date{}

\begin{document}
\maketitle
\pagestyle{main}

\begin{abstract}
In this article, we extend the computation of topological Hochschild homology ($\THH$) of the Adams summand $\ell$ of $p$-local connective complex topological $K$-theory ($ku$) to $ku$ itself. We leverage the relation $u^{p-1} = v_1$, where $u$ is a generator of $ku_*$ and $v_1$ is a generator of $\ell_*$, and we consider the cofiber of the multiplication by $v_1$ in $ku$, denoted $ku/v_1$. We use the morphism between the Bockstein spectral sequences of the multiplication by $v_1$ computing $\THH_*(\ell)$ and $\THH_*(ku)$; we develop a general technique using what we term a \emph{gathered spectral sequence} that allows us to explore the relationship between the Bockstein spectral sequences for the multiplications by $v_1$ and $u$, from which we derive a complete computation of $\THH_*(ku)$. Our method is not only applicable to this specific problem but may also prove useful in other computations.
\end{abstract}

\tableofcontents

\section{Introduction}

\subsection{Overview}

Let $ku$ be the $p$-localized connective cover of the complex topological $K$-theory spectrum $KU$, and $\ell$ its Adams summand. In this article, we lift the computation by Angeltveit, Hill and Lawson in \cite{angeltveit2010topological} of $\THH_*(\ell)$ to a computation of $\THH_*(ku)$. The weak equivalence
\begin{equation*}
  \THH(A;B) \simeq B\wedge_A\THH(A)
\end{equation*}
and a filtration of $B$ yields a filtration of $\THH(A; B)$. If this filtration arises from multiplication by an element of $B$, we thus get a Bockstein spectral sequence computing $\THH_*(A;B)$; the Bockstein spectral sequence computed in \cite{angeltveit2010topological} has the form
\begin{equation*}
  \THH_*(\ell;H\Z_{(p)})\otimes P(v_1) \Rightarrow \THH_*(\ell)
\end{equation*}
with $P(v_1) = \ell_*$ a polynomial algebra.

The same construction can be used over $ku$, with a Bockstein spectral sequence of the form
\begin{equation*}
  \THH_*(ku;H\Z_{(p)})\otimes P(u) \Rightarrow \THH_*(ku).
\end{equation*}
However, these two spectral sequences cannot be compared directly using the map of $\mathbb{S}$-modules $\ell\rightarrow ku$, since the generators of $\THH_*(\ell;H\Z_{(p)})$ that support the differentials in the first spectral sequence map to zero in the second.

To lift the computation to $ku$, we need a more subtle technique. The map of $\mathbb{S}$-modules $\ell \to ku$ sends $v_1$ to $u^{p-1}$, so that there is a map of $\mathbb{S}$-modules between the cofiber of the multiplication by $u^{p-1} = v_1$ in $ku$, denoted $ku/v_1$, and in the corresponding cofiber in $\ell$. We also have a weak equivalence $\ell/v_1 \simeq H\Z_{(p)}$, and the morphism
\begin{equation*}
  \THH_*(\ell;H\Z_{(p)}) \rightarrow \THH_*(ku; ku/v_1)
\end{equation*}
is non-trivial. Moreover, multiplication by $v_1$ in $ku$ yields a third Bockstein spectral sequence, of the form
\begin{equation*}
  \THH_*(ku; ku/v_1) \otimes P(v_1) \Rightarrow \THH_*(ku)
\end{equation*}
which can then be compared \emph{via} the morphism with the Bockstein spectral sequence computing $\THH_*(\ell)$. The next step is to compare both Bockstein spectral sequences computing $\THH_*(ku)$, coming from the multiplication by $u$ and $v_1$. We can represent their underlying exact couples:
\begin{center}
  \begin{tikzcd}[column sep=5pt]
    & & & & \operatorname{T}_*(ku; ku/v_{1}) \ar[dll] & & & & \\
    \dots \ar[rr] & & \operatorname{T}_*(ku) \ar[rr, "u"] & & \operatorname{T}_*(ku) \ar[rr, "u"] \ar[ld] & & \operatorname{T}_*(ku) \ar[rr] \ar[ld] \ar[ull] & & \dots \\
    & & & \operatorname{T}_*(ku;H\Z_{(p)}) \ar[lu] & & \operatorname{T}_*(ku;H\Z_{(p)}) \ar[lu] & & &
  \end{tikzcd}
\end{center}
Here we write $\operatorname{T}$ for $\THH$, and we have chosen $p = 3$, so that $v_1 = u^2$; the top exact couple is for multiplication by $v_1$, the bottom one for multiplication by $u$. The top exact couple is obtained by skipping steps in the filtration providing the bottom exact couple. We call the spectral sequence associated to the top exact couple a \emph{gathered spectral sequence} of the bottom spectral sequence, since the differentials from multiple pages of the bottom sequence are gathered into a single page in the top sequence.

In Section~\ref{chap:topologicalhh}, we provide the necessary background on topological Hochschild homology, in particular the different spectral sequences that we will use in our computations. In Section~\ref{chap:reviewthhkul}, we review some computations related to $\THH(\ell)$ and $\THH(ku)$ that constitute the input of the present work. We provide a correspondence between the differentials in a gathered spectral sequence and the differentials in the base spectral sequence in Section~\ref{chap:truncated}, and use it to compute the Bockstein spectral sequence of the multiplication by $u$ converging to $\THH_*(ku)$ in Section~\ref{chap:thhkucomplet}. The technique developed for this computation is general and could be used in other computations where some power of a multiplicative element is better understood than the element itself.

We note that another computation of $\THH_*(ku)$, using different techniques, was carried out in \cite{lee2022integral}.

\subsection{Notations and conventions}

We will use the following notations to describe various algebras:
\begin{itemize}
\item $P(x)$ is a polynomial algebra over a generator $x$,
\item $P_n(x)$ is a truncated polynomial algebra at height $n$, that is the quotient of $P(x)$ by the relation $x^n=0$,
\item  $\Gamma(x)$ is a divided power algebra, which is generated additively by the divided power of $x$, denoted $\gamma_i x$ for any $i\geq 0$, and with the multiplicative relations:
  \begin{equation*}
    \gamma_i x\cdot \gamma_j x = \binom{i+j}{i} \gamma_{i+j} x,
  \end{equation*}
\item $E(x)$ is an exterior algebra, which means $P_2(x)$.
\end{itemize}
The base ring for these algebras will be determined in most cases by the context in which they appear. When computing homology with coefficients in $\F_p$ or modulo $p$ homotopy, the base ring will be $\F_p$. When computing homology with coefficients in $\Z$, $\Z_{(p)}$ or $\Z_p$ (the integers, the $p$-localized integers or the $p$-completed integers), it will be $\Z$, $\Z_{(p)}$ or $\Z_p$. When computing $\THH$, it will be the base ring for the coefficient spectrum. If we need to specify the base ring, we will denote it by a subscript: $P_\Q(x)$, $E_\Q(x)$, etc.

When writing spectral sequences, we will use tensor products $\otimes$ of these algebras. One of these tensor products will be written $\botimes$, it will separate the algebras generated by classes whose bidegree lies on the $x$-axis -- on the left of $\botimes$ -- and those generated by classes whose bidegree lies on the $y$-axis -- on the right of $\botimes$.

When we write generators in the form $v_0^h\sigma u\mu_N$, we will use the conventions
\begin{equation*}
  v_0^0\sigma u \mu_N = \sigma u\mu_N \quad v_0^1\sigma u \mu_N = v_0\sigma u\mu_N \quad \sigma u \mu_0 = \sigma u.
\end{equation*}
The same convention applies with $\sigma u$ replaced by $\sigma v_1$.

\subsection{Acknowledgements}

The content of this article is part of my PhD thesis, which I am grateful to have done under the supervision of Christian Ausoni at the University Paris 13. I would also like to thank my current institution, the Sino-French Institute of the Renmin University of China, for supporting my work.

\section{Topological Hochschild homology}
\label{chap:topologicalhh}

In this section, we will define topological Hochschild homology and some of the tools, mostly spectral sequences, that we will later use in our computation.

The spectral sequence that appears with the first definition of topological Hochschild homology by Bökstedt in \cite{bokstedtthhzzp} is of the following type:
\begin{equation*}
  HH_*(H_*(R;\F_p)) \Rightarrow H_*(\THH(R);\F_p)
\end{equation*}
where $R$ is a ring spectrum and $HH_*$ is the Hochschild homology. This was used to compute $\THH_*(\Z_p)$ and $\THH_*(\F_p)$ in \cite{bokstedtthhzzp}.

The existence of an algebra structure on $\THH(R)$ allows the construction of various Bockstein spectral sequences associated to the multiplication by some element of the algebra; the exact couple of a Bockstein spectral sequence is obtained from the cofiber sequence of the multiplication by the chosen element. Although multiple Bockstein spectral sequences can be constructed from an algebra, they must all compute the same thing. That fact yields a computation of $\THH_*(\ell)$ in \cite{angeltveit2010topological} by making the Bockstein spectral sequences for multiplication by $p$ and $v_1$ compete.

Brun's spectral sequence computes $\THH$ of a ring $A$ with coefficients in an $A$-algebra $B$ from $\THH$ of $B$ with coefficients in a generalized $\operatorname{Tor}$ group in the sense of \cite{elmendorf1997rings}. In \cite{brun2000topological}, that spectral sequence is introduced to compute $\THH_*(\Z/p^n)$. Modern categories of spectra allow us to express this spectral sequence as an Atiyah-Hirzebruch spectral sequence, as done in \cite{honing2020brun} to compute $V(1)_*\THH(ku)$ and $V(0)_*\THH(K(\F_q);\Z_p)$, where $V(0)$ and $V(1)$ are Smith-Toda complexes. One finds that the input of Brun's spectral sequence is often simpler than its output; moreover, this input can often itself be computed using Brun's spectral sequence, yielding further simplifications.

Let $R$ be a commutative $S$-algebra; we work in the category $\M_R$ of $R$-modules of \cite{elmendorf1997rings}, from which most of our definitions are taken.

\subsection{Simplicial spectra and their realization}

Let $\Delta$ be the simplex category, whose objects are the ordered sets of integers $[n]=\{0,\dots,n\}$ and whose morphisms are the order-preserving maps.

\begin{definition}
  A simplicial $R$-module is a functor $F:\Delta^{op}\rightarrow \M_R$.

  For such a functor, its geometric realization, denoted $|F|$, is the coend
  \begin{equation*}
 \int^\Delta F\wedge (\Delta_\bullet)_+ 
\end{equation*}
that is the coend of the functor $\Delta^{op}\times \Delta\rightarrow \M_R$ that sends $(n,m)$ to $F(n)\wedge (\Delta_m)_+$, where $\Delta_\bullet$ is the topological simplex, viewed as a functor $\Delta\rightarrow Top$.

Similarly, a simplicial based space is a functor $F:\Delta^{op}\rightarrow Top_*$, and its geometric realization $|F|$ is the coend of the functor $F\wedge (\Delta_\bullet)_+$.
\end{definition}

The geometric realization, as a coend, is in fact a coequalizer, and thus commutes with colimits. Other useful properties of the geometric realization are:
\begin{proposition}[X.1.3 of \cite{elmendorf1997rings}]
  \begin{itemize}
  \item For a simplicial based space $X_\bullet$, there is a natural isomorphism
    \begin{equation*}
      \Sigma^\infty |X_\bullet| \cong |\Sigma^\infty X_\bullet|.
    \end{equation*}
  \item For a simplicial based space $X_\bullet$ and a simplicial $R$-module $Y_\bullet$, a simplicial $R$-module $Y_\bullet\wedge X_\bullet$ can be obtained by composing the diagonal $\Delta^{op}\rightarrow \Delta^{op}\times\Delta^{op}$ with the functor $\Delta^{op}\times\Delta^{op}\rightarrow \M_R$ sending $(n,\,m)$ to $Y_n\wedge X_m$, and there is a natural isomorphism
    \begin{equation*}
      |Y_\bullet \wedge X_\bullet| \cong |Y_\bullet| \wedge |X_\bullet| .
    \end{equation*}
  \item For two simplicial spectra $Y_\bullet$ and $Z_\bullet$, again using the diagonal structure, there is a natural isomorphism
    \begin{equation*}
      |Y_\bullet \wedge Z_\bullet| \cong |Y_\bullet| \wedge |Z_\bullet| .
    \end{equation*}  
  \end{itemize}
\end{proposition}

A useful example of a simplicial $R$-module is given by the bar construction: 
\begin{definition}[IV.7.2 of \cite{elmendorf1997rings}]
  For an $S$-algebra $R$ with multiplication $\phi$ and unit $\eta$, a right $R$-module $M$ with action $\mu$ and a left $R$-module $N$ with action $\nu$, the bar construction of $(M,R,N)$ is the simplicial $S$-module $B_\bullet(M,R,N)$ whose $n$-th simplicial level is
  \begin{equation*}
 B_n(M,R,N) = M\wedge R^{\wedge n} \wedge N 
\end{equation*}
whose $i$-th face map is
\begin{equation*}
  d_i =
  \begin{cases}
    \mu \wedge id_R^{\wedge n - 1} \wedge id_N & \text{ if} i = 0 \\
    id_M \wedge id_R^{\wedge i -1} \wedge \phi \wedge id_R^{\wedge n - i - 1} \wedge id_N & \text{ if} 0 < i < n \\
    id_M \wedge id_R^{\wedge n - 1} \wedge \nu & \text{ if} i = n
  \end{cases}
\end{equation*}
and whose $i$-th degeneracy map is $s_i = id_M \wedge id_R^{\wedge i} \wedge \eta \wedge id_R^{\wedge n - i} \wedge id_N$.

  Denote by $B(M,R,N)$ the realization $|B_\bullet(M,R,N)|$.
\end{definition}

\begin{proposition}[IV.7.5 of \cite{elmendorf1997rings}]
  For $M$ a cell $R$-module and $N$ any $R$-module, there is a natural weak equivalence
  \begin{equation*}
 B(M,R,N) \simeq M\wedge_R N .
\end{equation*}
\end{proposition}

If $R$ is commutative, $A$ is an $R$-algebra and $M$ and $N$ are right and left $A$-modules, one can also form the bar construction $B^R_\bullet(M,A,N)$ by replacing all the smash products by smash products over $R$. In that case:
\begin{proposition}[X.1.2 and XII.1.2 of \cite{elmendorf1997rings}]
  There is a natural weak equivalence $B^R(A,A,N)\simeq N$.
\end{proposition}

In the next section, we define topological Hochschild homology as a simplicial spectrum.

\subsection{Simplicial definition of THH and consequences}

Let $R$ be a cofibrant commutative $S$-algebra, $A$ a cofibrant $R$-algebra, $M$ an $(A,A)$-bimodule. Let
\begin{equation*}
 \phi:A\wedge_R A\rightarrow A \mbox{ and } \eta:R\rightarrow A 
\end{equation*}
be the multiplication and unit of $A$. Let
\begin{equation*}
 \xi_\ell:A\wedge_R M\rightarrow M \mbox{ and } \xi_r:M\wedge_R A\rightarrow M 
\end{equation*}
be the left and right action of $A$ on $M$. Let
\begin{equation*}
 \tau:M\wedge_R A^{\wedge n} \wedge_R A \rightarrow A\wedge_R M \wedge_R A^{\wedge n} 
\end{equation*}
be the map cyclically permuting the factors. In the following all the smash products are over $R$.

\begin{definition}[IX.2.1 of \cite{elmendorf1997rings}]
  The topological Hochschild homology of $A$ with coefficients in $M$ is the realization, denoted $\THH^R(A;M)$, of the simplicial $R$-module $\THH^R(A;M)_\bullet$ whose $n$-th simplicial level is given by
  \begin{equation*}
 \THH^R(A;M)_n = M\wedge_R A^{\wedge n} 
\end{equation*}
  with $i$-th face map given by $\xi_r\wedge id^{n-1}$ if $i=0$, $id\wedge id^{i-1} \wedge \phi \wedge id^{n-i-1}$ if $0<i<n$, $(\xi_\ell \wedge id^{n-1})\circ \tau$ if $i=n$; and with $i$-th degeneracy map given by $id\wedge id^i \wedge \eta \wedge id^{n-1}$.
\end{definition}

When $M = A$, this construction is also called the cyclic bar construction.

When working over $R=S$, we will drop the $^S$ from the notation. When $M=A$, we will write $\THH^R(A)=\THH^R(A;A)$. When $A$ is commutative, topological Hochschild homology has the following structure:
\begin{proposition}[IX.2.2 of \cite{elmendorf1997rings}]
  Let $A$ be a commutative $R$-algebra. Then $\THH^R(A)$ is naturally a commutative $A$-algebra with unit map the inclusion of the 0-th simplicial level $A\rightarrow \THH^R(A)$; $\THH^R(A;M)$ is a $\THH^R(A)$-module.
\end{proposition}

From the cited properties of the geometric realization with respect to the smash product, and by viewing $M$ as a constant simplicial spectrum, one can see that:
\begin{proposition}
  When $A$ is commutative and $M$ is a symmetric $(A,A)$-bimodule, there is a natural isomorphism of simplicial $R$-modules
  \begin{equation*}
 M \wedge_A \THH^R(A)_\bullet \cong \THH^R(A;M)_\bullet 
\end{equation*}
  and thus a natural isomorphism of $R$-modules
  \begin{equation*}
 M \wedge_A \THH^R(A) \cong \THH^R(A;M) .
\end{equation*}
\end{proposition}
A key application is the Smith-Toda complex $V(0)$ (the modulo $p$ sphere), for which $V(0)\wedge H\Z \cong V(0)\wedge H\Z_p \cong H\F_p$, so
\begin{equation*}
  V(0)\wedge\THH(A;H\Z) \cong V(0)\wedge\THH(A;H\Z_p)\cong \THH(A;H\F_p).
\end{equation*}

The simplicial construction of THH can also be linked with the bar construction. For an $R$-algebra $A$, let $A^e= A\wedge_R A^{op}$ be the enveloping algebra of $A$, where $A^{op}$ is the $R$-algebra obtained by composing the multiplication $A\wedge_R A \rightarrow A$ of $A$ with the map permuting the two factors $A\wedge_R A \rightarrow A\wedge_R A$.
\begin{proposition}[IX.2.4 and IX.2.5 of \cite{elmendorf1997rings}]\label{prop:thhandbar}
  There is a natural isomorphism
  \begin{equation*}
 \THH^R(A;M) \cong M\wedge_{A^e} B^R(A,A,A) 
\end{equation*}
  that gives a natural weak equivalence
  \begin{equation*}
 \THH^R(A;M) \simeq M \wedge_{A^e} A 
\end{equation*}
  when $M$ is a cell $A^e$-module.
\end{proposition}

Thus, we could have defined $\THH^R(A;M)$ as the derived smash product $M\wedge^L_{A^e} A$, which is the second definition proposed in \cite{elmendorf1997rings}.

\subsection{Spectral sequences computing THH}

The original result of Brun was the following:
\begin{theorem}[theorem 6.2.10 of \cite{Bru00}]
  When $R\rightarrow A$ is a ring homomorphism between (discrete) commutative rings, there is a multiplicative spectral sequence:
  \begin{equation*}
 E^2_{n,m} = \THH_n(HA;H\operatorname{Tor}^R_m(A,A)) \Rightarrow \THH(HR;HA) .
\end{equation*}
\end{theorem}

That result was generalized by Höning in \cite{honing2020brun}.
\begin{theorem}[theorem 1.1 of \cite{honing2020brun}]
  \label{the:brunsshoning}
  Let $A$ be a cofibrant commutative $S$-algebra and $B$ be a connective cofibrant commutative $A$-algebra. Let $E$ be an $S$-ring spectrum. Then there is a multiplicative spectral sequence of the form
  \begin{equation*}
    E^2_{n,m} = \THH_n(B;HE^S_m(B\wedge_A B)) \Rightarrow E^S_{n+m}(\THH(A;B))
  \end{equation*}
  with differentials
  \begin{equation*}
    d^r_{_n,m}:E^r_{n,m}\rightarrow E^r_{n-r, m+r - 1}.
  \end{equation*}
\end{theorem}

Topological Hochschild homology can also be computed using a Künneth spectral sequence.
\begin{proposition}[Lemma 2.2 and Corollary 2.3 of \cite{angeltveit2010topological}] \label{prop:simplebockstedtss}
  Suppose $R\rightarrow Q$ is a map of $S$-algebras and $M$ is a $(Q,R)$-bimodule, given the structure of an $(R,R)$-bimodule by pullback. Then there is a weak equivalence
  \begin{equation*}
 \THH(R;M) \simeq M\wedge^L_{Q\wedge R^{op}} Q 
\end{equation*}
  and thus a Künneth spectral sequence
  \begin{equation*}
 \operatorname{Tor}_{*,*}^{Q_*R^{op}}(M_*,Q_*) \Rightarrow \THH_*(R;M) .
\end{equation*}
\end{proposition}

The last spectral sequence we will use in our computation is the Bockstein spectral sequence. We now formulate its definition in the context of topological Hochschild homology.

Assume that $A$ is a commutative $R$-algebra, that $M$ is a connective, symmetric $(A,A)$-bimodule and that $m \in \pi_n(A)$ with $n \geq 0$. By acting on $M$, $m$ defines a map of $(A,A)$-bimodules $m:\Sigma^n M \rightarrow M$. Let $M/m$ be the cofiber
\begin{equation*}
  \begin{tikzcd}
    \Sigma^n M \rar["m"] & M \rar & M/m.
  \end{tikzcd}
\end{equation*}
We can define an exact couple from the tower of spectra with cofibers
\begin{equation*}
  \begin{tikzcd}
    \dots \rar["m"] & \Sigma^{2n} M \rar["m"] \dar & \Sigma^n M \rar["m"] \dar & M \dar \rar["id"] & \dots  \\
    & \Sigma^{2n}M/m & \Sigma^n M/m & M/m & 
  \end{tikzcd}
\end{equation*}
after smashing with $\THH(A)$ over $A$. The spectral sequence associated to this exact couple is called the Bockstein spectral sequence of the multiplication by $m$, and is of the form
\begin{equation*}
  E^1_{*,*} = \THH_*(A; M/m) \botimes P(m) \Rightarrow \THH_*(A; M)
\end{equation*}
where $\THH_*(A; M/m)$ is in bidegree $(*, 0)$, $m$ is in bidegree $(0, n)$ and the differentials are $m$-linear and of the form $d_i(\alpha) = m^i \beta$, thus in bidegree $|d_i| = (- ni  - 1, ni)$.

\begin{proposition}[Bockstein spectral sequence]
  If $A$ and $M$ are connective, the spectral sequence 
  \begin{equation*}
    \THH_*(A;M/m)\botimes P(m) \Rightarrow \THH_*(A;M).
  \end{equation*}
  is strongly convergent.
\end{proposition}
\begin{proof}
  By Proposition~\ref{prop:simplebockstedtss}, there is a Künneth spectral sequence
  \begin{equation*}
    \operatorname{Tor}^{A_*A^{op}}(A_*, A_*) \Rightarrow \THH_*(A).
  \end{equation*}
  When $A$ is connective, $\operatorname{Tor}^{A_*A^{op}}(A_*, A_*)$ is connective and so is $\THH(A)$. Our tower of spectra is
  \begin{equation*}
    \begin{tikzcd}
      \dots \rar["m"] & \Sigma^n \THH(A;M) \rar["m"] \dar & \THH(A;M) \dar \\
      & \Sigma^n \THH(A;M/m) & \THH(A;M/m)
    \end{tikzcd}
  \end{equation*}
  and $\lim_{k\in\Z}(\Sigma^{kn}\THH(A;M))_* = 0$ because of the suspension. Thus by Proposition~\ref{prop:convergencessnicecase} the spectral sequence is strongly convergent.
\end{proof}

\subsection{Smashing localizations and THH}

Let $R$ be a cofibrant commutative $S$-algebra, $A$ a cofibrant $R$-algebra and $M$ an $(A,A)$-bimodule. Let $E$ be a cell $R$-module. We will study the Bousfield localization at $E$, whose definition and useful properties can be found in chapter VIII of \cite{elmendorf1997rings}. We suppose that the Bousfield localization at $E$ of an $R$-module is smashing, that is, the localization of any $R$-module $X$, denoted $X_E$, can be realized as $R_E \wedge_R X$ where $R_E$ is the Bousfield localization of $R$ at $E$. We can construct $R_E$ to be an $R$-algebra and the localization map $\lambda: R\rightarrow R_E$ to be an algebra map. Then the localization map of $A$
\begin{equation*}
  \begin{tikzcd}
    \lambda: A \rar["\simeq"] & R\wedge_R A \rar["\lambda\wedge id"] & R_E\wedge_R A
  \end{tikzcd}
\end{equation*}
can be seen to be an $R$-algebra map, where the multiplication on $R_E\wedge_R A$ is
\begin{equation*}
  \begin{tikzcd}
    R_E\wedge_R A\wedge_R R_E\wedge_R A \rar["id\wedge\tau\wedge id"] & R_E\wedge_R R_E \wedge_R A\wedge_R A \rar["\mu\wedge\nu"] & R_E \wedge_R A
  \end{tikzcd}
\end{equation*}
where $\tau$ switches the two factors and $\mu$ and $\nu$ are the multiplications. Similarly, $M_E$ can be given both the structure of an $(A,A)$-bimodule, making $\lambda$ an $(A,A)$-bimodule map, and an $(A_E,A_E)$-bimodule structure.

\begin{proposition}\label{prop:localizationthh}
  If the above conditions are met, then there are weak equivalences
  \begin{equation*}
    \THH^R(A;M)_E \cong \THH^R(A;M_E) 
  \end{equation*}
  and 
  \begin{equation*}
    \THH^R(A;M_E) \simeq \THH^R(A_E;M_E) .
  \end{equation*}
\end{proposition}
\begin{proof}
  $\THH^R(A;M)_E$ can be seen to be the realization of the simplicial object $R_E\wedge_R \THH^R(A;M)_\bullet$, which is also $\THH^R(A;M_E)_\bullet$. This yields the first weak equivalence.

  The map $\lambda: R_E \rightarrow R_E\wedge_R R_E$ as defined above is an $E$-equivalence between $E$-local $R$-modules, and thus a weak equivalence.  Define a simplicial map
  \begin{equation*}
 \THH^R(A;M_E)_\bullet \rightarrow \THH^R(A_E;M_E)_\bullet
\end{equation*}
  such that on the $n$-th simplicial level we have:
  \begin{center}
    \begin{tikzcd}
      M_E\wedge_R A^{\wedge n} = R_E\wedge_R M \wedge_R A^{\wedge n} \dar["\simeq"] \\
      R_E\wedge_R R^{\wedge n}\wedge_R M \wedge_R A^{\wedge n} \dar["id\wedge\lambda^n\wedge id"] \\
      R_E\wedge_R R_E^{\wedge n}\wedge_R M \wedge_R A^{\wedge n} \dar["\tau"] \\
      R_E\wedge_R M \wedge_R (R_E\wedge_R A)^{\wedge n}=M_E\wedge_R A_E^{\wedge n} .
    \end{tikzcd}
  \end{center}
  Each of these maps is a weak equivalence, so by taking a suitable cellular replacement and by Theorem X.1.2 of \cite{elmendorf1997rings}, we get a weak equivalence between the realizations.
\end{proof}

\section{Review of the results on $\ell$ and $ku$}
\label{chap:reviewthhkul}

In this section we review some results about $\THH_*(ku;H\Z_{(p)})$, $\THH_*(\ell;H\Z_{(p)})$, $\THH_*(\ell)$ and the periodic spectra $\THH(KU)$ and $\THH(L)$. We first give a computation of $\THH_*(ku;H\Z_{(p)})$ using the Brun spectral sequence in Section~\ref{section:thhkuhz}. The Bockstein spectral sequence \eqref{ss:l}, computing $\THH_*(\ell)$, was established in \cite{angeltveit2010topological}, and we review this result in Section~\ref{section:thhl}.

Our $q$-cofibrant commutative $S$-algebra model for the connective complex $K$-theory spectrum $ku$ will be that of Theorem VII.4.3 of \cite{elmendorf1997rings}; regardless of the choice of model, the $E_\infty$ structure on $ku$ can be seen to be unique (see \cite{baker2008uniqueness}). We fix a prime $p$ and write $ku$ for the $p$-localized connective complex $K$-theory and $\ell$ its Adams summand. We obtain an $S$-algebra structure on the localization using the result on Bousfield localization stated in Proposition VIII.1.8 of \cite{elmendorf1997rings}.

\subsection{Topological Hochschild homology of $ku$ with coefficients in $H\Z_{(p)}$}
\label{section:thhkuhz}

In this section, we compute $\THH_*(ku;H\Z_{(p)})$ with $p$ an odd prime. When $p=2$, we have $ku = \ell$; results about $\ell$ are in Section~\ref{section:thhl}. We use the Brun spectral sequence of Theorem~\ref{the:brunsshoning}:
\begin{equation}
  E^2_{p,q}=\THH_p(H\Z_{(p)};H\pi_q(H\Z_{(p)}\wedge_{ku}H\Z_{(p)}))\Rightarrow \THH_{p+q}(ku;H\Z_{(p)}) \tag{$u_\Z$}\label{ss:uZ}.
\end{equation}

The Künneth spectral sequence can be used to compute the coefficients.
\begin{proposition}
  There is an isomorphism
  \begin{equation*}
    \pi_*(H\Z_{(p)}\wedge_{ku}H\Z_{(p)}) \cong E(\sigma u)
  \end{equation*}
  where $E(\sigma u)$ is an exterior algebra over $\Z_{(p)}$ on the generator $\sigma u$ of degree 3.
\end{proposition}
\begin{proof}
  $\Z_{(p)}$ has a resolution as a free $ku_*$-module given by $E(\sigma u)$, with $\sigma u$ of bidegree $(1,2)$ and $d(\sigma u) = u$, so that $\operatorname{Tor}_{*,*}^{ku_*}(\Z_{(p)},\Z_{(p)})\cong E(\sigma u)$. Then the Künneth spectral sequence
  \begin{equation*}
    E^2_{p,q}=\operatorname{Tor}_{p,q}^{ku_*}(\Z_{(p)},\Z_{(p)})\Rightarrow \pi_{p+q}(H\Z_{(p)}\wedge_{ku}H\Z_{(p)})
  \end{equation*}
  collapses for bidegree reasons with no possible extensions.
\end{proof}

The $E^2$-page of our Brun spectral sequence will then consist of two copies of $\THH_*(H\Z_{(p)};H\Z_{(p)})=\THH_*(H\Z_{(p)})$. Topological Hochschild homology of $H\Z$ was computed by Bökstedt in \cite{bokstedtthhzzp}:
\begin{equation*}
  \THH_k(H\Z)=
  \begin{cases}
      \Z & \mbox{if $k=0$} \\
      0 & \mbox{if $k\geq 2$ is even} \\
      \Z/n & \mbox{if $k=2n-1\geq 2$.} \\
    \end{cases}
  \end{equation*}
  Since localization at $p$ is smashing, we have
  \begin{equation*}
    \THH_k(H\Z_{(p)})=
    \begin{cases}
      \Z_{(p)} & \mbox{if $k=0$} \\
      0 & \mbox{if $k\geq 2$ is even} \\
      \Z/\nu(n) & \mbox{if $k=2n-1\geq 2$} \\
    \end{cases}
  \end{equation*}
  where $\nu$ is the $p$-adic valuation.
  Let $\mu_{n}$ be a generator of the copy of $\Z/\nu(n)$ in degree $2n-1$. If $n$ is not divisible by $p$, then $\mu_{n} = 0$. We will also use the convention $\mu_0 = 1$ in our formulas.

  \begin{proposition}
    When $p$ is an odd prime, the spectral sequence \eqref{ss:uZ} collapses at the $E^2$-page. There are no extensions, and
    \begin{equation*}
      \THH_*(ku; H\Z_{(p)}) \cong \THH_*(H\Z_{(p)}) \otimes E(\sigma u)
    \end{equation*}
    over $\Z_{(p)}$ with $\sigma u$ in degree 3.
  \end{proposition}
  \begin{proof}
    For bidegree reasons, the only possible non-zero differentials are the $d^4$ between $\mu_{n+2}$ and $\sigma u\mu_{n}$. But if $p\geq 3$ divides $n+2$, it cannot divide $n$, so that at least one of $\mu_{n+2}$ or $\sigma u \mu_{n}$ is zero, and the spectral sequence collapses.

    Since there is at most one generator in each degree, no extensions are possible.
  \end{proof}

\subsection{Topological Hochschild homology of $\ell$}
\label{section:thhl}

In this section, we will review the results of \cite{angeltveit2010topological} on $\THH_*(\ell)$, for any prime $p$. The first spectral sequence, denoted \eqref{ss:lZ}, is a Brun spectral sequence:
\begin{gather}
  E^2_{*,*} = \THH_*(H\Z_{(p)}; H(H\Z_{(p)}\wedge_\ell H\Z_{(p)})_*) \cong \THH_*(H\Z_{(p)})\botimes E(\sigma v_1)  \notag \\
  \Rightarrow \THH_*(\ell;H\Z_{(p)}) \tag{$\ell_\Z$}\label{ss:lZ}
\end{gather}
The generators have bidegrees 
\begin{equation*}
  \begin{aligned}
    & |\mu_{kp}| = (2kp-1,0) \mbox{, $k\geq 1$ the generators of $\THH_*(H\Z_{(p)})$} \\
    & |\sigma v_1| = (0,2p-1)\\
    & |v_1| = (0,2(p-1)).
  \end{aligned}
\end{equation*}
and $|d^r| = (-r, r - 1)$.

The second spectral sequence, denoted \eqref{ss:l}, is a Bockstein spectral sequence:
\begin{equation}
  E^1_{*,*} = \THH_*(\ell;H\Z_{(p)})\botimes P(v_1) \Rightarrow \THH_*(\ell) \tag{$\ell$}. \label{ss:l}
\end{equation}
Here, the elements of $\THH_*(\ell; H\Z_{(p)})$ are in bidegrees $(*, 0)$, $v_1$ has bidegree $(0, 2(p-1))$ and the differentials have bidegree $|d^r| = (-2r(p-1) - 1 , 2r(p-1))$.

Let $\mathcal{R}$ be a (discrete) commutative ring, and $\mathcal{A}$ an $\mathcal{R}$-algebra. When necessary, we will use $x\cdot y$ for the $\mathcal{R}$-action of $x\in \mathcal{R}$ on $y\in \mathcal{A}$, and $xy$ for the product of $x,\,y\in \mathcal{A}$.  From \cite{angeltveit2010topological}, Proposition 3.4, which computes $\THH_*(\ell;H\Z_{(p)})$ we can deduce:
\begin{proposition}\label{prop:diffin4}
  All the differentials in \eqref{ss:lZ} are given by the formulas:
  \begin{equation*}
    d^{2p}(\mu_{(k+1)p})=p^{\nu(k)}\cdot\sigma v_1\mu_{kp}
  \end{equation*}
  up to a unit, where $k\geq 1$ and $\nu$ is the $p$-adic valuation.

  There is an extension given by $p\mu_p = \sigma v_1$.
\end{proposition}

The spectral sequence \eqref{ss:l} is also computed in \cite{angeltveit2010topological}. We will use the following notations:
\begin{equation*}
  \THH_*(\ell;H\Z_{(p)}) \cong \Z_{(p)}\{1,\, \mu_p\}\oplus\bigoplus_{k\geq 2} \faktor{\Z}{p^{\nu(k)}}\{v_0\mu_{kp},\, \sigma v_1 \mu_{kp}\} .
\end{equation*}
Here from the Brun spectral sequence \eqref{ss:lZ}  we have $\sigma v_1 = p\cdot \mu_p$ and $v_0\mu_{kp}$ is a class in $\THH$ represented by $p\cdot\mu_{kp}\in E^\infty$. As in \cite{angeltveit2010topological}, we write $v_0^k \alpha$ for a lift of a class $p^k \cdot \alpha$ from the spectral sequence.

\begin{theorem}[Theorem 6.4 of \cite{angeltveit2010topological}] \label{prop:diffinell}
  The differentials in \eqref{ss:l} are given by the formulas:
  \begin{equation*}
    d^{p^{n+1}+\dots +p}(p^{n}\cdot v_0\mu_{(k+1)p^{n+2}}) = k v_1^{p^{n+1}+\dots+p}\sigma v_1\mu_{kp^{n+2}},\; k\geq 0,\; n\geq 0
  \end{equation*}
  up to a unit and the differentials are $v_1$-linear.
\end{theorem}

There are extensions at the end of this spectral sequence. We now state the result with our notations:
\begin{theorem}[sections 6.2 and 6.3 of \cite{angeltveit2010topological}]
  $\THH_*(\ell)$ is a quotient of the $\Z_{(p)}[v_1]$-module
  \begin{multline}
    \Z_{(p)}[v_1]\{1,\,\sigma v_1,\, v_0^n\mu_{p^{n+1}},\,n\geq 0\} \\
    \oplus \Z_{(p)}[v_1]\{v_0^h\sigma v_1\mu_{ap^n},\, n\geq 2,\, a\geq 1,\, \mbox{$a$ not divisible by $p$},\, h \geq 0\}
  \end{multline}
  by the relations in the non-torsion part:
  \begin{itemize}
  \item $p\cdot \mu_p = \sigma v_1$,
  \item $p\cdot v_0^n\mu_{p^{n+1}} = v_1^{p^n}v_0^{n-1}\mu_{p^n}$ for any $n\geq 1$,
  \end{itemize}
  and the relations in the torsion part:
  \begin{itemize}
  \item $v_0^{h} \sigma v_1 \mu_{ap^n} = 0$ for any $a \geq 1$ and $n \geq 2$, $a$ not divisible by $p$, and $h \geq n -1$,
  \item $v_1^{p^{n-h-1}+p^{n-h-2}+\dots+p }\cdot v_0^h\sigma v_1 \mu_{ap^n} = 0$ for any $a \geq 1$ and $n \geq 2$, $a$ not divisible by $p$ and  $0 \leq h \leq n - 2$,
  \item $p\cdot \sigma v_1 \mu_{(bp+p-1)p^n} = v_0\sigma v_1\mu_{(bp+p-1)p^n}+v_1^{p^n}v_0^{\nu(b)}\sigma v_1 \mu_{bp^{n+1}}$ for any $b\geq 1$ and $n \geq 2$,
  \item $p\cdot v_0^h\sigma v_1 \mu_{ap^n} = v_0^{h+1}\sigma v_1 \mu_{ap^n}$ for any $a \geq 1$, $n \geq 2$, $a$ not divisible by $p$, and any $1 \leq h \leq n - 2 $, or $h=0$ not in the previous case.
  \end{itemize}
  
  The degrees are:
  \begin{equation*}
    \begin{aligned}
      & |\mu_{kp}| = 2kp - 1 \\
      & |\sigma v_1| = 2p - 1 \\
      & |v_0| = 0 \\
      & |v_1| = 2(p-1)
    \end{aligned}
  \end{equation*}
  and $\nu$ is the $p$-adic valuation.
\end{theorem}

In order to lift this computation to the Bockstein spectral sequence \eqref{ss:u}, computing $\THH_*(ku)$, one must compare the sequences by means other than the map induced by the inclusion $\ell\rightarrow ku$, since $\sigma v_1\in \THH_{2p-1}(\ell;H\Z_p)$ should be compared to $u^{p-2}\sigma u$, which is not a class in $\THH_{2p-1}(ku;H\Z_p)$. A solution is to consider the cofiber of the multiplication by $v_1$:
\begin{center}
  \begin{tikzcd}
    \Sigma^{2p-1}ku \ar[r, "v_1"] & ku \ar[r] & ku/v_1.
  \end{tikzcd}
\end{center}
This is done in Section~\ref{chap:thhkucomplet}. Section~\ref{chap:truncated} develops the correspondence used to compare the two Bockstein spectral sequences computing $\THH_*(ku)$, the first associated with $u$ and the second with $v_1$.

\section{Spectral sequences from towers of spectra}
\label{chap:truncated}

Our vocabulary concerning spectral sequences will follow Boardman's in \cite{boardman1999conditionally}. We will work in the homotopy category of a category of spectra. The underlying category of spectra could be Boardman's spectra (see \cite{adams1974stable} or \cite{switzer2017algebraic}), or $S$-module from \cite{elmendorf1997rings}. What we really use is that we have a triangulated category, with a homotopy functor to the category of graded groups that produces long exact sequences from the triangles, with some uniqueness on the maps between two triangles (arising from the uniqueness up to homotopy of the maps between cofiber sequences).

We study spectral sequences arising from a tower of spectra indexed by $\Z$:
\begin{center}
  \begin{tikzcd}
    ... \arrow[r] & Y_{n+1} \arrow[r] & Y_{n} \arrow[r] & Y_{n-1} \arrow[r] & ...
  \end{tikzcd}
\end{center}
Let $Y_\infty$ be the limit of the tower and $Y_{-\infty}$ be the colimit. For any $a$ and $b$ integers or $\pm\infty$ with $a\leq b$, let $Y_a^b$ be the cofiber of the map $Y_b\rightarrow Y_a$. For each $n\in\Z$, the cofiber sequence:
\begin{center}
  \begin{tikzcd}
    Y_{n+1} \arrow[r] & Y_{n} \arrow[r] & Y_n^{n+1}
  \end{tikzcd}
\end{center}
gives a long exact sequence in homotopy. Pasting each of these sequences defines an unrolled exact couple, and a spectral sequence.

To allow (weak) convergence of the spectral sequence, we quotient the tower of spectra by the limit. Note that having a zero limit is a necessary condition for convergence, but not a sufficient one. To study this quotient, we need to discuss the maps between these cofibers.

\subsection{The octahedral axiom and consequences}

The octahedral axiom is assumed true in any triangulated category. Here we will use it in the homotopy category of spectra, which is triangulated by virtue of being the homotopy category of a stable model category.

\begin{axiom}[Octahedral]
  Let $A\rightarrow B\rightarrow C$, $A\rightarrow D\rightarrow E$ and $B\rightarrow D\rightarrow F$ be triangles such that the diagram
  \begin{center}
    \begin{tikzcd}
      A \rar \dar["id"] & B \dar \\
      A \rar & D \\
    \end{tikzcd}
  \end{center}
  commutes. Then there are six triangles and a commutative diagram:
  \begin{center}
    \begin{tikzcd}
      A \rar \dar["id"] & B \dar \rar & C \dar \\
      A \rar \dar & D \rar \dar & E \dar \\
      * \rar & F \rar["id"] & F \\
    \end{tikzcd}
  \end{center}
  where $*$ is the zero-object of the category.
\end{axiom}
Note that in the specific case of the stable homotopy category, the maps $C\rightarrow E$ and $E\rightarrow F$ are unique, and thus are unique up to homotopy in the category of spectra.

Our first lemma is a reformulation of this axiom with our notation.
\begin{lemma}\label{prop:cofiber1}
  Let $a\leq b \leq c$ be integers or $\pm\infty$. There is a morphism of cofiber sequences, and a commutative diagram:
  \begin{center}
    \begin{tikzcd}
      Y_c \rar \dar["id"] & Y_b \rar \dar & Y_b^c \dar \\
      Y_c \rar & Y_a \rar & Y_a^c \\
    \end{tikzcd}
  \end{center}
   There is a cofiber sequence:
  \begin{center}
    \begin{tikzcd}
      Y_b^c \arrow[r] & Y_a^c \arrow[r] & Y_a^b
    \end{tikzcd}
  \end{center}
  and a weak equivalence $f: Y_a^b \rightarrow Y_a^b$ making the following diagram commute.
  \begin{center}
    \begin{tikzcd}
      Y_b \rar \dar & Y^c_b \dar \\
      Y_a \rar \dar & Y^c_a \dar \\
      Y^b_a \rar["f"]  & Y^b_a 
    \end{tikzcd}
  \end{center}
\end{lemma}

We can conclude the following, which ensures that our spectral sequences can converge to the homotopy groups of their colimit.
\begin{proposition}\label{prop:quotientlimitss}
  For any integers $a\leq b$, the cofiber of $Y_b^{\infty} \rightarrow Y_a^\infty$ is $Y_a^b$.
  Moreover, the two towers of spectra
  \begin{center}
    \begin{tikzcd}
      ... \arrow[r] & Y_{n+1}^\infty \arrow[r] & Y_{n}^\infty \arrow[r] & Y_{n-1}^\infty \arrow[r] & ...
    \end{tikzcd}
  \end{center}
  \begin{center}
    \begin{tikzcd}
      ... \arrow[r] & Y_{n+1} \arrow[r] & Y_{n} \arrow[r] & Y_{n-1} \arrow[r] & ...
    \end{tikzcd}
  \end{center}
  induce isomorphic spectral sequences, beginning from the $E^1$ pages.
\end{proposition}
\begin{proof}
  This follows from Lemma~\ref{prop:cofiber1}: we have a morphism of exact couple induced by the diagrams
  \begin{center}
    \begin{tikzcd}
      Y_{n+1} \rar \dar & Y_n \rar \dar & Y_n^{n+1} \dar["\simeq"] \\
      Y_{n+1}^\infty \rar & Y_n^\infty \rar & Y_n^{n+1} \\
    \end{tikzcd}
  \end{center}
  that is an isomorphism on the $E^1$ pages. The induced morphisms on the derived exact couples are then automatically isomorphisms on the following pages, and thus we have two isomorphic spectral sequences.
\end{proof}

This proposition will be used with towers of spectra such that for some $m\in\Z$ and for all $k\geq m$, the maps $Y_{k+1}\rightarrow Y_{k}$ are isomorphisms -- that is, $Y_m$ is the limit of the tower; and thus $\infty$ will be replaced by $m$. In fact, we will mostly deal with towers quotiented by their limits, and we will need another version of the octahedral axiom.

In the following, whenever $i\leq j \leq k$ are integers or $\pm\infty$, the map $Y_j^k\rightarrow Y_i^k$ is the map coming from the morphism between the cofiber sequences $Y_k\rightarrow Y_j \rightarrow Y_j^k$ and $Y_k\rightarrow Y_i \rightarrow Y_i^k$, and the map $Y_i^k \rightarrow Y_i^j$ is from the cofiber sequence $Y_j^k\rightarrow Y_i^k \rightarrow Y_i^j$ of Lemma~\ref{prop:cofiber1}. Both are unique up to homotopy.

\begin{lemma}\label{prop:cofiber2}
  Let $a\leq b \leq c \leq d$ be integers or $\pm\infty$. There are commutative diagrams, both containing six cofiber sequences:
  \begin{center}
    \begin{tikzcd}
      Y_c^d \rar \dar["id"] & Y_b^d \rar \dar & Y_b^c \dar \\
      Y_c^d \rar \dar & Y_a^d \rar \dar & Y_a^c \dar \\
      * \rar & Y_a^b \rar["\simeq"] & Y_a^b \\
    \end{tikzcd}
    \hspace{2cm}
    \begin{tikzcd}
      Y_c^d \rar \dar & Y_a^d \rar \dar["id"] & Y_a^c \dar \\
      Y_b^d \rar \dar & Y_a^d \rar \dar & Y_a^b \dar \\
      Y_b^c\rar & * \rar & \Sigma Y_b^c \\
    \end{tikzcd}
  \end{center}
\end{lemma}
\begin{proof}
  The left one follows directly from the octahedral axiom. The right one must be shifted once in the horizontal direction using $\Sigma$ to have the same form as the octahedral axiom. The maps can be seen to be the canonical ones since they are unique up to homotopy.
\end{proof}

\subsection{Truncated and gathered spectral sequences}

For any spectrum $\Gamma$, write $\Gamma_*=\pi_*(\Gamma)$ for its homotopy groups. The tower
\begin{center}
  \begin{tikzcd}
    ... \arrow[r] & Y_{n+1}^\infty \arrow[r] & Y_{n}^\infty \arrow[r] & Y_{n-1}^\infty \arrow[r] & ...
  \end{tikzcd}
\end{center}
gives a spectral sequence of the form
\begin{equation*}
  (\mathcal{B}) : E^1 = \bigoplus_{n\in\Z} (Y_n^{n+1})_* \Rightarrow (Y_{-\infty}^\infty)_*.
\end{equation*}

In the cases that are of interest to us, this spectral sequence will be strongly convergent by Theorem 6.1 of \cite{boardman1999conditionally}.
\begin{proposition} \label{prop:convergencessnicecase}
  If for all $n\in\Z$, the spectra $Y^\infty_n$ are connective, then $(\mathcal{B})$ is a half-plane spectral sequence with exiting differentials (in the sense of II.6 of \cite{boardman1999conditionally}); if, moreover, its limit $\lim_{n\in\Z}(Y^\infty_n)_*$ is zero, then it is strongly convergent toward its colimit $(Y^\infty_{-\infty})_*$.
\end{proposition}

For any integers $a\leq b$, we can truncate the tower at $a$ and $b$, and thus the spectral sequence $(\mathcal{B})$.
Let $X$ be the tower such that:
\begin{equation*}
  X_n=
  \begin{cases}
      Y_b^\infty \mbox{ if } n\geq b \\
      Y_a^\infty \mbox{ if } n\leq a \\
      Y_n^\infty \mbox{ otherwise}
    \end{cases}
  \end{equation*}
with identity maps when necessary and maps induced by the original tower. This defines a truncated spectral sequence:
\begin{equation*}
  (\mathcal{T}_a^b) : E^1 = \bigoplus_{a\leq n < b} (Y_n^{n+1})_* \Rightarrow (Y_{a}^b)_*.
\end{equation*}
Note that the tower obtained by quotienting by the limit has components:
\begin{equation*}
  X'_n=
  \begin{cases}
    Y_b^b\simeq * & \mbox{ if } n\geq b \\
    Y_a^b & \mbox{ if } n\leq a \\
    Y_n^b &  \mbox{ otherwise.}
  \end{cases}
\end{equation*}

For any strictly increasing map $\phi:\Z\rightarrow\Z$, consider the tower whose $n$-th level is $Y_{\phi(n)}^\infty$ and whose maps are the composition of corresponding maps in the original tower. This defines what we call a gathered spectral sequence:
\begin{equation*}
  ({}^\phi \mathcal{B}) : E^1 = \bigoplus_{n\in\Z} (Y_{\phi(n)}^{\phi(n+1)})_* \Rightarrow (Y_{-\infty}^\infty)_*.
\end{equation*}

If $\phi$ is the multiplication by 2, the pages of $({}^\phi \mathcal{B})$ are gathered two-by-two; the first differential $d^1$ of $({}^\phi \mathcal{B})$ contains information about the $d^2$ and $d^3$ of $(\mathcal{B})$, the second about $d^4$ and $d^5$, etc.

The application we have in mind is comparing the Bockstein spectral sequences associated to the multiplication by $u$ and $v_1$ that compute $\THH_*(ku)$:
\begin{equation*}
  \begin{gathered}
    \eqref{ss:u}: \quad \THH_*(ku; H\Z_{(p)}) \otimes P(u) \Rightarrow \THH_*(ku) \\
    \eqref{ss:v1}: \quad \THH_*(ku; ku/v_1) \otimes P(v_1) \Rightarrow \THH_*(ku)
  \end{gathered}
\end{equation*}
Here, \eqref{ss:v1} is a gathered spectral sequence obtained from \eqref{ss:u} with the map $\phi(n) = (p-1) n$ (since $v_1 = u^{p-1}$). We will use our results on gathered spectral sequences of Section~\ref{chap:thhkucomplet} to compare \eqref{ss:u} and \eqref{ss:v1}.

If one wants to compute $(Y_{-\infty}^\infty)_*$, this gives two ways to do it: computing $(\mathcal{B})$, or computing each $(Y_{\phi(n)}^{\phi(n+1)})_*$ by means of $(\mathcal{T}_{\phi(n)}^{\phi(n+1)})$ and thereafter computing $({}^\phi \mathcal{B})$. These two computations are not independent. Let us represent our spectral sequences graphically with the following grading: the $n$ in $(Y_n)_*$ (the filtration degree) is the $y$-coordinate, and the $x$-coordinate is such that $*=x+y$. With such a bidegree, the differentials will have $|d^r|=(-r-1,r)$ when we let the exact couple given by the tower of spectra be the $E^1$ page.  We will draw first quadrant spectral sequences, but our results apply to whole plane spectral sequences.

\begin{figure}
  \begin{tikzpicture}[scale=0.85]
    \tikzset{->/.style={-{Latex}}}
    \draw[step=1cm, gray, thin, dotted] (-0.2,-0.2) grid (12.2,6.2);
    \draw[->] (-0.2,0) -- (12.4,0) node[below] {x};
    \draw[->] (0,-0.2) -- (0,6.4) node[left] {y};
    \foreach \x in {0,...,12}
    {
      \foreach \y in {0,...,6}
      {
        \node (\x\y) at (\x,\y) {$\bullet$};
      }
    }
    \foreach \y in {0,...,6}
    {
      \pgfmathsetmacro\ypu{int(\y+1)}
      \node () at (-0.6,\y) {$(Y_\y^\ypu)_*$};
    }
    \draw[->] (41) --  node[above right] {$d^1$} (22);
    \draw[->] (103) --  node[above right] {$d^2$} (75);
    \draw[->] (80) --  node[above right] {$d^4$} (34);
  \end{tikzpicture}
  \caption{\label{fig:ss-B} Example of the spectral sequence $(\mathcal{B})$.}
\end{figure}
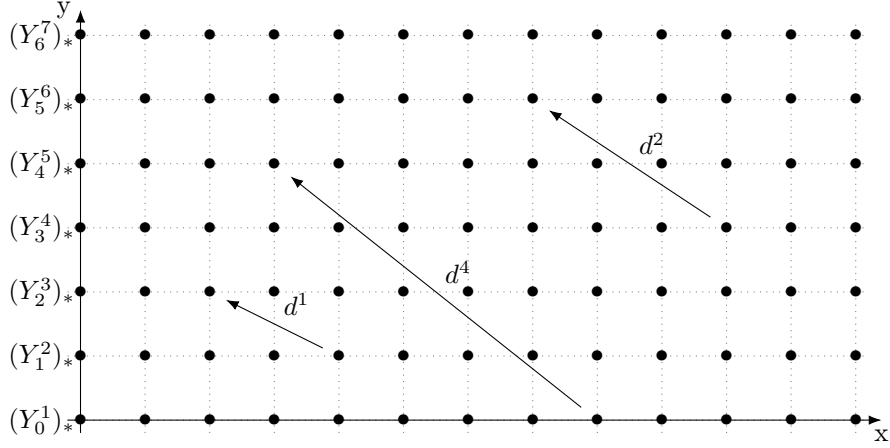

For each of figures~\ref{fig:ss-B}, \ref{fig:ss-T0}, \ref{fig:ss-T4} and \ref{fig:ss-Bphi}, a $\bullet$ represents a copy of a field $\F$ on the $E^1$-page, and the $\bullet^n$ in Figure~\ref{fig:ss-Bphi} represents $n$ copies of $\F$. On the Figure~\ref{fig:ss-B} we have shown 3 non-zero differentials of different size. We will choose our function $\phi:\Z \rightarrow \Z$ such that $\phi(0)=0$, $\phi(1)=3$ and $\phi(2)=7$. Our first result  is that the $d^1$ and $d^2$ shown will respectively be seen in $(\mathcal{T}_0^4)$ and $(\mathcal{T}_4^7)$, as seen in Figure~\ref{fig:ss-T0} and Figure~\ref{fig:ss-T4}. Conversely, having such differentials in $(\mathcal{T}_0^4)$ or $(\mathcal{T}_4^7)$ will give rise to a differential in $(\mathcal{B})$. This is the content of Theorem~\ref{prop:totalandsmall}.

\begin{figure}
  \begin{tikzpicture}[scale=0.85]
    \tikzset{->/.style={-{Latex}}}
    \draw[step=1cm, gray, thin, dotted] (-0.2,-0.2) grid (12.2,6.2);
    \draw[->] (-0.2,0) -- (12.4,0) node[below] {x};
    \draw[->] (0,-0.2) -- (0,6.4) node[left] {y};
    \foreach \x in {0,...,12}
    {
      \foreach \y in {0,...,2}
      {
        \node (\x\y) at (\x,\y) {$\bullet$};
      }
    }
    \foreach \y in {0,...,2}
    {
      \pgfmathsetmacro\ypu{int(\y+1)}
      \node () at (-0.6,\y) {$(Y_\y^\ypu)_*$};
    }
    \draw[->] (41) --  node[above right] {$d^1$} (22);
  \end{tikzpicture}
  \caption{\label{fig:ss-T0} The spectral sequence $(\mathcal{T}_0^3)$ corresponding to the $(\mathcal{B})$ of Figure~\ref{fig:ss-B}.}
\end{figure}

\begin{figure}
  \begin{tikzpicture}[scale=0.85]
    \tikzset{->/.style={-{Latex}}}
    \draw[step=1cm, gray, thin, dotted] (-0.2,-0.2) grid (12.2,6.2);
    \draw[->] (-0.2,0) -- (12.4,0) node[below] {x};
    \draw[->] (0,-0.2) -- (0,6.4) node[left] {y};
    \foreach \x in {0,...,12}
    {
      \foreach \y in {3,...,6}
      {
        \node (\x\y) at (\x,\y) {$\bullet$};
      }
    }
    \foreach \y in {3,...,6}
    {
      \pgfmathsetmacro\ypu{int(\y+1)}
      \node () at (-0.6,\y) {$(Y_\y^\ypu)_*$};
    }
    \draw[->] (103) --  node[above right] {$d^2$} (75);
  \end{tikzpicture}
  \caption{\label{fig:ss-T4} The spectral sequence $(\mathcal{T}_3^7)$ corresponding to the $(\mathcal{B})$ of Figure~\ref{fig:ss-B}.}
\end{figure}

However, the differential $d^4$ is too long and is ``jumping'' from the area covered by $(\mathcal{T}_0^3)$ to that covered by $(\mathcal{T}_3^7)$, and thus is not visible in either of the truncated spectral sequences. The $d^4$ differential will be visible in $({}^\phi \mathcal{B})$, as we will prove in Theorem~\ref{prop:totaltotruncated}; in Figure~\ref{fig:ss-Bphi}, we see that it gives a $d^1$ between the class in $(Y_0^3)_*$ represented by its source, and the class in $(Y_3^7)_*$ represented by its target.

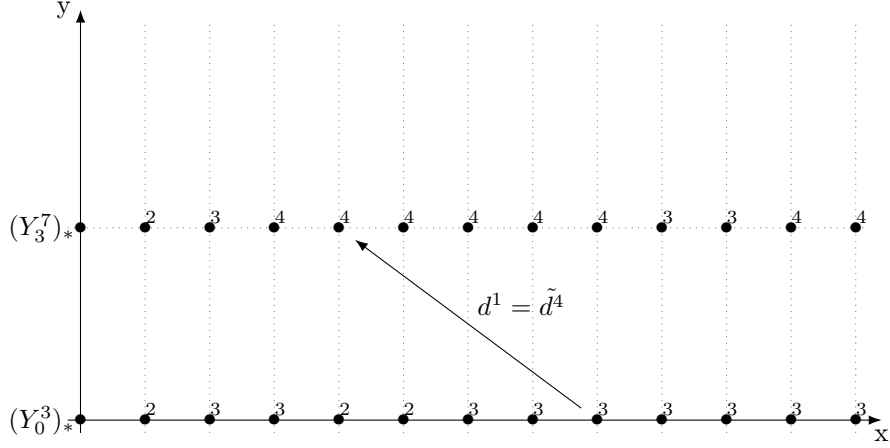
\begin{figure}
  \begin{tikzpicture}[scale=0.85]
    \tikzset{->/.style={-{Latex}}}
    \foreach \x in {1,...,12}
    {
      \draw[gray, thin, dotted] (\x,-0.2) to (\x, 6.2);
    }
    \draw[gray, thin, dotted] (-0.2,3) to (12.2,3);
    \draw[->] (-0.2,0) -- (12.4,0) node[below] {x};
    \draw[->] (0,-0.2) -- (0,6.4) node[left] {y};
    \foreach \x in {1,...,12}
    {
      \node (b\x) at (\x,0) {$\bullet$};
      \node () at (\x+.1,0.1) {\ifthenelse{\x=1 \OR \x=4 \OR \x=5}{$^2$}{$^3$}};
    }
    \foreach \x in {2,...,12}
    {
      \node (h\x) at (\x,3) {$\bullet$};
      \node () at (\x+.1,3.1) {\ifthenelse{\x=2 \OR \x=9 \OR \x=10}{$^3$}{$^4$}};
    }
    \node (b0) at (0,0) {$\bullet$};
    \node (h0) at (0,3) {$\bullet$};
    \node (h1) at (1,3) {$\bullet$};
    \node () at (1.1,3.1) {$^2$};
    \draw[->] (b8) --  node[above right] {$d^1=\widetilde{d^4}$} (h4);
    \node () at (-0.6,0) {$(Y_0^3)_*$};
    \node () at (-0.6,3) {$(Y_3^7)_*$};
  \end{tikzpicture}
  \caption{\label{fig:ss-Bphi} The spectral sequence $({}^\phi \mathcal{B})$ corresponding to the $(\mathcal{B})$ of Figure~\ref{fig:ss-B}.}
\end{figure}

Finally, Theorem~\ref{prop:truncatedtototal} deals with the case of transferring a differential of $({}^\phi \mathcal{B})$ into $(\mathcal{B})$, and Theorem~\ref{prop:nulldiffs} deals with the zero differentials in $(\mathcal{B})$ and $({}^\phi \mathcal{B})$.

Consider an unrolled exact couple:
\begin{center}
  \begin{tikzcd}[column sep = small]
    ... \ar[rr, "i"] & & A_{n+1} \ar[rr, "i"] & & A_n \ar[rr, "i"] \ar[dl, "j"] & & ... \\
     & & & E_n^1 \ar[ul, "k"] & & &
  \end{tikzcd}
\end{center}
For $r\geq 0$, let $Z_n^r$ and $B_n^r$ be the groups of $r$-cycles and of $r$-boundaries in $E_n^1$, that is:
\begin{equation*}
  \begin{gathered}
    Z_n^r = k^{-1}(Im(i^{r-1}:A_{n+r}\rightarrow A_{n+1})) \\
    B_n^r = j(Ker(i^{r-1}:A_n\rightarrow A_{n-r+1})).
  \end{gathered}
\end{equation*}
We let $E^r$ be the quotient $Z^{r}/B^{r}$ for $r\geq 1$, and the differential $d^r$ will be a map $E^r_n\rightarrow E^r_{n+r}$. We will write ${}^\phi Z_n^r$ and ${}^\phi B_n^r$ for the $r$-cycles and $r$-boundaries in the spectral sequence $({}^\phi \mathcal{B})$ to distinguish them from those in $(\mathcal{B})$.
\begin{definition}\label{prop:diffdiagram}
  For $x\in E^r_n$ and $y\in E^r_{n+r}$, we write $d^r(x)=y$ when for some $\bar{x}\in Z^{r}_n$ representing $x$ in the quotient and some $\bar{y}\in Z^{r}_{n+r}$ representing $y$, $k(\bar{x})$ can be lifted $r-1$ times through $i$, and the image under $j$ of the $(r-1)$-th lift is $\bar{y}$.
\end{definition}
We can visualize this in the exact couple diagram:
\begin{equation}\label{differentialdiagram}
  \begin{tikzcd}[column sep = small]
   A_{n+r+1} \ar[rr, "i"] & & A_{n+r} \ar[rr, "i"] \ar[dl, "j"] & & ... \ar[rr, "i"] & & A_{n+1} \ar[rr, "i"] & & A_n \ar[dl, "j"] \\
   & E_{n+r}^1 \ar[ul, "k"] & & &   & & & E_n^1 \ar[ul, "k"] & \\
  & & \alpha \ar[rr, mapsto] \ar[dl, mapsto] & & ... \ar[rr, mapsto] & & i^{r-1}(\alpha) & & \\
   & \bar{y} & & &   & & & \bar{x} \ar[ul, mapsto] &
  \end{tikzcd}
\end{equation}
Moreover, $y \neq 0$ if and only if $r$ is maximal, that is $k(\bar{x})$ is not in the image of $i^r$.

We now describe how the differentials in the spectral sequences $(\mathcal{B})$, $({}^\phi \mathcal{B})$ and $(\mathcal{T}_{\phi(n)}^{\phi(n+1)})$ are interlinked.

First, we see how a differential in $(\mathcal{B})$ short enough to fit in $(\mathcal{T}_{\phi(n)}^{\phi(n+1)})$ will occur.
\begin{theorem} \label{prop:totalandsmall}
  Let $n$, $r$ and $N$ be integers such that $\phi(N)\leq n\leq n+r < \phi(N+1)$, and let $x\in Z^r_n$ and $y\in Z^r_{n+r}$ in $(\mathcal{B})$.

  Then  $d^r(x)=y$ in $(\mathcal{B})$ if and only if $d^r(x)=y$ in $(\mathcal{T}_{\phi(N)}^{\phi(N+1)})$,
  where $x$ and $y$ stand for the quotients in the respective $E^r$-pages of the two spectral sequences.
\end{theorem}
\begin{proof}
  This is seen directly in Diagram~\eqref{differentialdiagram}. Note that the cycles are not in general the same between $(\mathcal{B})$ and $(\mathcal{T}_{\phi(N)}^{\phi(N+1)})$, but here we have $r < \phi(N+1) - \phi(N)$ so that the $r$-cycles are indeed the same.
\end{proof}

We then need a technical lemma to describe the longer differentials. In the following, for any $-\infty \leq a \leq b \leq c \leq \infty$, the maps
\begin{equation*}
  (Y^c_b)_* \to (Y^c_a)* \to (Y^b_a)_* \to (Y^c_b)_{*-1}
\end{equation*}
are the unique maps coming from the cofiber sequence of Lemma~\ref{prop:cofiber1}. Whenever $c = \infty$ and $b = a + 1$, we write $i$ for the map $(Y^\infty_{a+1})_* \to (Y^\infty_a)_*$, so that for any integers $a \leq b$, we write $i^{b - a}$ for the map $(Y^\infty_b)_* \to (Y^\infty_a)_*$. When necessary, we will give names to the maps not of these forms using other letters.

\begin{lemma}\label{lem:technicallemmafordiff}
  For integers $a \leq b \leq c$, if the commutative diagram
  \begin{center}
    \begin{tikzcd}
            (Y_b^{b+1})_* \ar[d, "f"]  & (Y_b^{c})_* \ar[d, "e"] \ar[r] \ar[l, "p"] & (Y_{a}^{c})_* \ar[d] \\
      (Y_{b+1}^{\infty})_{*-1}  & (Y_{c}^{\infty})_{*-1} \ar[l, "i^{c-b-1}"] \ar[r, "id"] & (Y_{c}^{\infty})_{*-1}
    \end{tikzcd}
  \end{center}
  can be populated with classes
  \begin{center}
    \begin{tikzcd}
      x \ar[d, mapsto] & & \\
      i^{c-b-1}(\beta) & \beta \ar[l, mapsto] \ar[r, mapsto] & \beta 
    \end{tikzcd}
  \end{center}
  then there exist lifts
  \begin{center}
    \begin{tikzcd}
      x \ar[d, mapsto] & \tilde{x}-h(u) \ar[d, mapsto] \ar[r, mapsto] \ar[l, mapsto] & \hat{x} \ar[d, mapsto]  \\
      i^{c-b-1}(\beta) & \beta \ar[l, mapsto] \ar[r, mapsto] & \beta 
    \end{tikzcd}
  \end{center}
  where $u \in (Y^c_{b+1})_*$ and $h$ is the map $(Y^c_{b+1})_* \to (Y^c_b)_*$.
\end{lemma}
\begin{proof}
  The diagram of the statement is commutative because of Lemma~\ref{prop:cofiber2}, which can also be used to check that the following diagram is commutative and has exact rows and columns:
  \begin{center}
    \begin{tikzcd}
      & (Y_{b}^{\infty})_{*} \rar["id"] \dar &  (Y_{b}^{\infty})_{*} \dar & \\
      (Y_{b+1}^{c})_{*} \rar["h"] \dar["id"] & (Y_{b}^{c})_{*} \rar["p"] \dar["e"] & (Y_{b}^{b+1})_{*} \dar["f"] \rar["g"] & (Y_{b+1}^{c})_{*-1} \dar["id"] \\
      (Y_{b+1}^{c})_{*} \rar["\delta"] & (Y_{c}^{\infty})_{*-1} \rar["i^{c-b-1}"] \dar["i^{c-b}"] & (Y_{b+1}^{\infty})_{*-1} \dar["i"] \rar & (Y_{b+1}^{c})_{*-1} \\
      & (Y_{b}^{\infty})_{*-1}\rar["id"] & (Y_{b}^{\infty})_{*-1}
      \end{tikzcd}
  \end{center}
  Here we can see that $x\in (Y_{b}^{b+1})_{*}$ can be lifted through $p$ to $(Y_b^{c})_*$: indeed, $f(x) = i^{c-b-1}(\beta)$ so $g(x) = 0$, and then there exists $\tilde{x}\in(Y_b^{c})_*$ such that $p(\tilde{x})=x$.

 In the central square of the diagram, we have chosen two elements in $(Y_{c}^\infty)_{*-1}$, $\beta$ and $e(\tilde{x})$, whose images by $i^{c-b-1}$ are equal. By pushing $\beta - e(\tilde{x})$ in the bottom square, we can see that it is in the image of $e$, and thus so is $\beta$. Write $\tilde{x}'$ such that $e(\tilde{x}')=\beta$, and $x'$ the image of $\tilde{x}'$ in $(Y_b^{b+1})_*$ by $p$.

  Now in the central square, $i^{c-b-1}(e(\tilde{x}-\tilde{x}'))=0$, so that there exists $u\in (Y_{b+1}^{c})_*$ with $\delta(u)=e(\tilde{x}-\tilde{x}')$. But the map $\delta$ factors through $(Y_{b}^{c})_*$ as $e\circ i$, and $h(u)\in (Y_{b}^{c})_*$ has image 0 in $(Y_b^{b+1})_*$ by $p$ since $u\in (Y_{b+1}^{c})_*$.

  Consider the element $\tilde{x}-h(u)\in (Y_{b}^{c})_*$:
  \begin{equation*}
    \begin{aligned}
      e(\tilde{x}-h(u))& = e(\tilde{x})-\delta(u) \\
      & = e(\tilde{x})-e(\tilde{x}-\tilde{x}') \\
      & = e(\tilde{x}') \\
      & = \beta
    \end{aligned}
  \end{equation*}
  \begin{equation*}
    \begin{aligned}
      p(\tilde{x}-h(u)) & = p(\tilde{x}) \\
      & = x.
    \end{aligned}
  \end{equation*}

  It remains to push $\tilde{x}-h(u)\in (Y_{b}^{c})_*$ into $(Y_{a}^{c})_*$, and we have:
   \begin{center}
    \begin{tikzcd}
      x \ar[d] & \tilde{x}-h(u) \ar[d, mapsto] \ar[r, mapsto] \ar[l, mapsto] & \hat{x} \ar[d, mapsto]  \\
      i^{c-b-1}(\beta) & \beta \ar[l, mapsto] \ar[r, mapsto] & \beta 
    \end{tikzcd}
  \end{center}
\end{proof}

We now describe how a longer differential in $(\mathcal{B})$ occurs in the gathered spectral sequence $({}^\phi \mathcal{B})$. We need the following definition:
\begin{definition}
  \label{def:represent}
  An infinite cycle $x\in (Y_n^{n+1})_*$ in the spectral sequence $(\mathcal{B})$ is said to \emph{represent} an element $\hat{x}$ of the target group $(Y^\infty_{-\infty})_*$ of $(\mathcal{B})$ if $x$ lifts through the map $(Y^\infty_n)_*\rightarrow (Y_n^{n+1})_*$ to an element $\tilde{x}\in (Y_n^\infty)_*$ whose image in $(Y_{-\infty}^\infty)_*$ is $\hat{x}$.
\end{definition}

\begin{theorem} \label{prop:totaltotruncated}
  Let $n$, $m$, $N$ and $M$ be integers such that
  \begin{equation*}
    \phi(N) \leq n < \phi(N+1) \leq \phi(M) \leq m < \phi(M+1)
  \end{equation*}
  and let $x\in Z^{m-n}_n$ and $y\in Z^{m-n}_{m}$ be classes in $(\mathcal{B})$ such that $d^{m-n}(x)=y\neq 0$.

  Then:
  \begin{itemize}
  \item $x$ is an infinite cycle in $(\mathcal{T}_{\phi(N)}^{\phi(N+1)})$, thus represents a class $\hat{x}\in (Y_{\phi(N)}^{\phi(N+1)})_*$,
  \item $y$ is an infinite cycle in $(\mathcal{T}_{\phi(M)}^{\phi(M+1)})$, thus represents a class $\hat{y}\in (Y_{\phi(M)}^{\phi(M+1)})_{*-1}$,
  \item There is a differential $d^{M-N}(\hat{x})=\hat{y}$ in $({}^\phi \mathcal{B})$.
  \end{itemize}
\end{theorem}
\begin{proof}
  We see that $x$ and $y$ are infinite cycles in the truncated spectral sequences using Diagram~\eqref{differentialdiagram}.

  The canonical maps assemble into a commutative diagram (it can be checked that each square is commutative using Lemma~\ref{prop:cofiber2}):
  \begin{equation} \label{eq:prooftotaltotruncated}
    \begin{tikzcd}
      (Y_n^{n+1})_* \ar[d, "f"]  & (Y_n^{\phi(N+1)})_* \ar[d, "e"] \ar[r] \ar[l, "p"] & (Y_{\phi(N)}^{\phi(N+1)})_* \ar[d] \\
      (Y_{n+1}^{\infty})_{*-1}  & (Y_{\phi(N+1)}^{\infty})_{*-1} \ar[l] \ar[r, "id"] & (Y_{\phi(N+1)}^{\infty})_{*-1} \\
      (Y_{m}^{\infty})_{*-1} \ar[u] \ar[d]  & (Y_{m}^{\infty})_{*-1} \ar[u] \ar[l, "id"] \ar[d] \ar[r] & (Y_{\phi(M)}^{\infty})_{*-1} \ar[u] \ar[d] \\
      (Y_{m}^{m+1})_{*-1}  & (Y_{m}^{\phi(M+1)})_{*-1} \ar[r] \ar[l] & (Y_{\phi(M)}^{\phi(M+1)})_{*-1}  \\
    \end{tikzcd}
  \end{equation}
  Note that $x\in (Y_n^{n+1})_*$ and $y\in (Y_{m}^{m+1})_{*-1}$.

  We use Definition~\ref{prop:diffdiagram}: having a differential $d^{m-n}(x)=y$ is having a class $\alpha\in (Y_{m}^\infty)_{*-1}$ with
   \begin{center}
    \begin{tikzcd}[row sep=tiny]
      (Y_{m}^{m+1})_{*-1} & (Y_{m}^{\infty})_{*-1} \ar[l] \ar[r] & (Y_{n+1}^{\infty})_{*-1} & (Y_n^{n+1})_* \ar[l] \\
      y & \alpha \ar[l, mapsto] \ar[r, mapsto] & i^{m-n-1}(\alpha) & x. \ar[l, mapsto]
    \end{tikzcd}
  \end{center}
  This is the left column of our diagram.

   We use Definition~\ref{def:represent} : having $y$ represent a class $\hat{y}\in (Y_{\phi(M)}^{\phi(M+1)})_*$ in $(\mathcal{T}_{\phi(M)}^{\phi(M+1)})$ is having an element $\tilde{y}\in(Y_{m}^{\phi(M+1)})_*$ such that
  \begin{center}
    \begin{tikzcd}[row sep=tiny]
      (Y_{m}^{m+1})_{*-1}  & (Y_{m}^{\phi(M+1)})_{*-1} \ar[r] \ar[l] & (Y_{\phi(M)}^{\phi(M+1)})_{*-1} \\
      y & \tilde{y} \ar[l, mapsto] \ar[r, mapsto] & \hat{y}.
    \end{tikzcd}
  \end{center}
  We choose $\hat{y}$ and $\tilde{y}$ by mapping $\alpha$ in the bottom right square.

  We now have populated our commutative diagram with the elements
  \begin{center}
    \begin{tikzcd}
      x \ar[d] & & \\
      i^{m-n-1}(\alpha) & i^{m-\phi(N+1)}(\alpha) \ar[l, mapsto] \ar[r, mapsto] & i^{m-\phi(N+1)}(\alpha)  \\
      \alpha \ar[d, mapsto] \ar[u, mapsto] & \alpha \ar[l, mapsto] \ar[u, mapsto] \ar[d, mapsto] \ar[r, mapsto] & i^{m-\phi(M)}(\alpha) \ar[d, mapsto] \ar[u, mapsto] \\
      y & \tilde{y} \ar[l, mapsto] \ar[r, mapsto] & \hat{y} \\
    \end{tikzcd}
  \end{center}

  We use Lemma~\ref{lem:technicallemmafordiff} with $a = \phi(N)$, $b = n$ and $c = \phi(N+1)$, and with $\beta = i^{m-\phi(N+1)}(\alpha)$, that is on our first two rows. We thus get lifts:
   \begin{center}
    \begin{tikzcd}
      x \ar[d] & \tilde{x}-h(u) \ar[d, mapsto] \ar[r, mapsto] \ar[l, mapsto] & \hat{x} \ar[d, mapsto]  \\
      i^{m-n-1}(\alpha) & i^{m-\phi(N+1)}(\alpha) \ar[l, mapsto] \ar[r, mapsto] & i^{m-\phi(N+1)}(\alpha)  \\
      \alpha \ar[d, mapsto] \ar[u, mapsto] & \alpha \ar[l, mapsto] \ar[u, mapsto] \ar[d, mapsto] \ar[r, mapsto] & i^{m-\phi(M)}(\alpha) \ar[d, mapsto] \ar[u, mapsto] \\
      y & \tilde{y} \ar[l, mapsto] \ar[r, mapsto] & \hat{y} \\
    \end{tikzcd}
  \end{center}

  The right column states that $d^{M-N}(\hat{x})=\hat{y}$ in $({}^\phi \mathcal{B})$.
\end{proof}

The next result describes how differentials in $({}^\phi \mathcal{B})$ have counterparts in $(\mathcal{B})$.
\begin{theorem} \label{prop:truncatedtototal}
  Let $N < M$ be integers and let $x\in {}^\phi Z^{M-N}_N$ and $y\in {}^\phi Z^{M-N}_M$ be classes in $({}^\phi \mathcal{B})$ such that $d^{M-N}(x)=y\neq 0$. For some unique $\phi(N)\leq n <\phi(N+1)$ and $\phi(M) \leq m < \phi(M+1)$, $x$ and $y$ are represented by $\check{x}\in (Y_n^{n+1})_*$ and $\check{y}\in (Y_m^{m+1})_{*-1}$ in the spectral sequence $(\mathcal{T}_{\phi(N)}^{\phi(N+1)})$ and $(\mathcal{T}_{\phi(M)}^{\phi(M+1)})$. Let $\check{x}$ and $\check{y}$ be fixed.
  
  Then there is a unique integer $n'$ such that  $\phi(N)\leq n \leq n' < \phi(N+1)$, and there is an element $x'\in (Y_{\phi(N)}^{\phi(N+1)})_*$ which is represented by $\check{x}'\in (Y_{n'}^{n'+1})_*$ in the spectral sequence $(\mathcal{T}_{\phi(N)}^{\phi(N+1)})$, that supports a differential $d^{M-N}(x')=y$ in $({}^\phi \mathcal{B})$, and such that there is a differential $d^{m-n'}(\check{x}')=\check{y} \neq 0$ in $(\mathcal{B})$. Moreover, $n'$ does not depend on the choice of the representative $\check{x}$ and $\check{y}$.
\end{theorem}
\begin{proof}
  We work again in Diagram~\eqref{eq:prooftotaltotruncated}. Note that $x\in (Y_{\phi(N)}^{\phi(N+1)})_{*}$ and that $y\in (Y_{\phi(M)}^{\phi(M+1)})_{*-1}$.

  First let's write $i^{m-\phi(N+1)}(\alpha)$ for the image of $x$ in $(Y_{\phi(N+1)}^\infty)_{*-1}$, with $m$ maximal for such a lift $\alpha$ in $(Y_m^\infty)_{*-1}$. Necessarily, $\phi(M)\leq m < \phi(M+1)$. By definition, the image of $i^{m-\phi(M)}(\alpha)$ in $(Y_{\phi(M)}^{\phi(M+1)})_{*-1}$ is $y$ up to a boundary of ${}^\phi B_M^{M-N}$; without loss of generality, we can suppose that it is $y$.

  We can then push $\alpha$ to get $\tilde{y}\in (Y_m^{\phi(M+1)})_{*-1}$ and $\check{y}\in (Y_m^{m+1})_{*-1}$. By definition, $n$ is such that $x$ can be lifted to $(Y_n^{\phi(N+1)})_*$ but not to $(Y_{n+1}^{\phi(N+1)})_*$. Denote $\tilde{x}$ such a lift and $\check{x}$ its non-zero image in $(Y_n^{n+1})_*$.

  Our diagram is populated as follows:
  \begin{center}
    \begin{tikzcd}
      \check{x} \ar[d] & \tilde{x} \ar[d, mapsto] \ar[r, mapsto] \ar[l, mapsto] & x \ar[d, mapsto]  \\
      i^{m-n-1}(\alpha) & i^{m-\phi(N+1)}(\alpha) \ar[l, mapsto] \ar[r, mapsto] & i^{m-\phi(N+1)}(\alpha)  \\
      \alpha \ar[d, mapsto] \ar[u, mapsto] & \alpha \ar[l, mapsto] \ar[u, mapsto] \ar[d, mapsto] \ar[r, mapsto] & i^{m-\phi(M)}(\alpha) \ar[d, mapsto] \ar[u, mapsto] \\
      \check{y} & \tilde{y} \ar[l, mapsto] \ar[r, mapsto] & y \\
    \end{tikzcd}
  \end{center}
  It is however possible that $i^{m-n-1}(\alpha)$ is zero.

  Let $n'$ be the largest integer such that $i^{m-n'}(\alpha) = 0 \in (Y_n^\infty)_{*-1}$. Since $i^{m-n-1}(\alpha)=f(\check{x})$, $i^{m-n}(\alpha)=0$ so $n\leq n'$. We now work in Diagram~\eqref{eq:prooftotaltotruncated} with $n$ replaced by $n'$: $i^{m-n'-1}(\alpha)$ can be lifted to $(Y_{n'}^{n'+1})_*$ since $i^{m-n'}(\alpha)=0$. Denote $\check{x}'$ such a lift. Again using Lemma~\ref{lem:technicallemmafordiff} on our first two rows we can construct classes $\tilde{x}'\in (Y_{n'}^{\phi(N+1)})_*$ and $x'\in (Y_{\phi(N)}^{\phi(N+1)})_*$ to complete the diagram and get the result.
\end{proof}

Note that with this level of generality, no stronger statement can be made: one may have to replace $\check{x}$ with $\check{x}'$ to get the differential in $(\mathcal{B})$. In fact, let us consider the tower of spectra such that:
\begin{equation*}
  Y_n =
  \begin{cases}
    * & \mbox{if $n\geq 3$} \\
    H\Z & \mbox{if $n=2$} \\
    * & \mbox{if $n=1$} \\
    \Sigma H\Z & \mbox{if $n\leq 0$}
  \end{cases}
\end{equation*}
and the integer function $\phi$ such that:
\begin{equation*}
  \phi(n) =
  \begin{cases}
    n & \mbox{if $n\leq 0$} \\
    n+1 & \mbox{if $n\geq 1$.}
  \end{cases}
\end{equation*}

We will draw the interesting part of the tower of spectra for each spectral sequence with the cofibers below. Note that with $(\mathcal{T}_0^2)$ we quotient the tower by the limit which is $Y_2$, and that we put between braces the name of a generator for the homotopy.

\begin{equation*}
 (\mathcal{B}): \quad
  \begin{tikzcd}
    Y_3 \rar & Y_2 \rar \dar & Y_1 \rar \dar & Y_0 \dar \\
    & Y_2^3 & Y_1^2 & Y_0^1 \\
    * \rar & H\Z \rar \dar & * \rar \dar & \Sigma H\Z \dar \\
    & H\Z\{\bar{y}\} & \Sigma H\Z\{\hat{x}'\} & \Sigma H\Z\{\hat{x}-\hat{x}'\} \\
  \end{tikzcd}
\end{equation*}
In $(\mathcal{B})$ there is a differential $d(\hat{x}')=\bar{y}$.

\begin{equation*}
 (\mathcal{T}_0^2): \quad
  \begin{tikzcd}
    Y_2^2 \rar & Y_1^2 \rar \dar & Y_0^2 \dar \\
    & Y_1^2 & Y_0^1 \\
    * \rar & \Sigma H\Z\{\bar{x}'\} \rar \dar & \Sigma H\Z\{\bar{x}'\}\vee \Sigma H\Z\{\bar{x}-\bar{x}'\} \dar \\
    & \Sigma H\Z\{\hat{x}'\} & \Sigma H\Z\{\hat{x}-\hat{x}'\} \\
  \end{tikzcd}
\end{equation*}
In $(\mathcal{T}_0^2)$ there is no non-zero differential.

\begin{equation*}
 ({}^\phi \mathcal{B}): \quad
  \begin{tikzcd}
    Y_3 \rar & Y_2 \rar \dar & Y_0 \dar \\
    & Y_2^3 & Y_0^2 \\
    * \rar & H\Z \rar \dar & \Sigma H\Z \dar \\
    & H\Z\{\bar{y}\} & \Sigma H\Z\{\bar{x}'\}\vee \Sigma H\Z\{\bar{x}-\bar{x}'\} \\
  \end{tikzcd}
\end{equation*}
In $({}^\phi \mathcal{B})$ there are differentials $d(\bar{x}')=\bar{y}$, and $d(\bar{x}-\bar{x}')=0$. But now, with slightly different notation from Theorem~\ref{prop:truncatedtototal}, we have a class $\bar{x}=(\bar{x}-\bar{x}')+\bar{x}'$ such that $d(\bar{x})=\bar{y}$ in $({}^\phi \mathcal{B})$, and that class is represented by $\hat{x}-\hat{x}'$ at the end of $(\mathcal{T}_0^2)$ since $\hat{x}'$ is of lower filtration. But in $(\mathcal{B})$, $d(\hat{x}-\hat{x}')=0$, the differential is really supported by $\hat{x}'$. Thus, we cannot get a better result. However, this will not be an issue in the application that follows, since we will be able to prove a better result for the Bockstein spectral sequences we will compute.

Statements can also be made regarding trivial differentials.
\begin{theorem} \label{prop:nulldiffs}
  \begin{enumerate}
  \item   Let $x\in (Y_{\phi(N)}^{\phi(N+1)})_*$ be an $(M-N)$-cycle in $({}^\phi \mathcal{B})$, that is $d^{i}(x)=0$ for $i\in\{1,\,\dots,\, M-N\}$. Then any  $\hat{x}\in (Y_n^{n+1})_*$ representing $x$ in $(\mathcal{T}_{\phi(N)}^{\phi(N+1)})_*$ (so that $\phi(N) \leq n < \phi(N+1)$) is such that $d^{m-n}(\hat{x})=0$ in $(\mathcal{B})$ for any m with $n < m < \phi(M+1)$.
  \item Let $\phi(N) \leq n < \phi(N+1) < m$. Let $\hat{x}\in (Y_n^{n+1})_*$ be an $(m-n)$-cycle in $(\mathcal{B})$. Then there exists a class  $x\in (Y_{\phi(N)}^{\phi(N+1)})_*$ represented by $\hat{x}$ in $(\mathcal{T}_{\phi(N)}^{\phi(N+1)})_*$ such that $x$ is an $(M-N)$-cycle in $({}^\phi \mathcal{B})$ for any $M$ such that $\phi(N+1) < \phi(M+1) \leq m$.
  \end{enumerate}
\end{theorem}
\begin{proof}
  Let $M'$ be such that $N < M' \leq M$ and let $m$ be such that $n < m$ and $\phi(M') \leq m < \phi(M' + 1)$. For the first claim, we work in Diagram~\eqref{eq:prooftotaltotruncated} with $M$ replaced by $M'$. 
  Our hypothesis implies that there exists $\tilde{x} \in (Y^{\phi(N+1)}_n)_*$ whose image in $(Y^{\phi(N+1)}_{\phi(N)})_*$ is $x$ and in $(Y^{n+1}_n)_*$ is $\hat{x}$; moreover, the image of $x$ in $(Y^\infty_{\phi(N+1)})_{*-1}$ is of the form $i^{\phi(M') - \phi(N+1)}(\alpha)$ for some $\alpha \in (Y^\infty_{\phi(M')})_{*-1}$, and $\alpha$ maps to zero in $(Y^{\phi(M'+1)}_{\phi(M')})_{*-1}$.

  Having image zero in $(Y^{\phi(M'+1)}_{\phi(M')})_{*-1}$ implies that $\alpha$ lifts to $\beta \in (Y^\infty_{\phi(M'+1)})_{*-1}$. Since $m < \phi(M' +1)$, $\beta$ maps to $\tilde{\beta} \in (Y^\infty_m)_{*-1}$ which is a lift of $\alpha$. We have populated the diagram as follows:
  \begin{center}
    \begin{tikzcd}
      \hat{x} \dar[mapsto] & \tilde{x} \lar[mapsto] \rar[mapsto] \dar[mapsto] & x \dar[mapsto] \\
      i^{m - n -1}(\tilde{\beta}) & i^{m - \phi(N+1)}(\tilde{\beta}) \lar[mapsto] \rar[mapsto] & i^{\phi(M') - \phi(N+1)}(\alpha) \\
      \tilde{\beta} \uar[mapsto] \dar[mapsto] & \tilde{\beta} \uar[mapsto] \dar[mapsto] \lar[mapsto] \rar[mapsto] & \alpha \uar[mapsto] \dar[mapsto] \\
      \gamma & \epsilon \lar[mapsto] \rar[mapsto] & 0
    \end{tikzcd}
  \end{center}
  with $\gamma \in (Y^{m+1}_m)_{*-1}$ and $\epsilon \in (Y^{\phi(M' + 1)}_m)_{*-1}$. We have to prove that $\gamma = 0$. But since $\tilde{\beta}$ comes from $\beta \in (Y^\infty_{\phi(M'+1)})_{*-1}$, $\tilde{\beta}$ lifts to $(Y^\infty_{m+1})_{*-1}$ and thus maps to $\gamma = 0$ in $(Y^{m+1}_m)_{*-1}$. Therefore, claim 1 is proven.

  For the second claim, we can populate Diagram~\eqref{eq:prooftotaltotruncated} as follows:
  \begin{center}
    \begin{tikzcd}
      \hat{x} \dar[mapsto] & & \\
      i^{m - n -1}(\alpha) & i^{m - \phi(N+1)}(\alpha) \lar[mapsto] \rar[mapsto] & i^{m - \phi(N+1)}(\alpha) \\
      \alpha \uar[mapsto] \dar[mapsto] & \alpha \uar[mapsto] \dar[mapsto] \lar[mapsto] \rar[mapsto] & i^{m - \phi(M)}(\alpha) \uar[mapsto] \dar[mapsto] \\
      0 & \epsilon \lar[mapsto] \rar[mapsto] & \gamma
    \end{tikzcd}
  \end{center}
  We use a similar argument to prove that $\gamma = 0$: since $\phi(M+1) \leq m$, $i^{m - \phi(M)}(\alpha) \in (Y^\infty_{\phi(M)})_{*-1}$ lifts to $i^{m - \phi(M+1)}(\alpha) \in (Y^\infty_{\phi(M+1)})_{*-1}$ so that it maps to $\gamma = 0$ in $(Y^{\phi(M+1)}_{\phi(M)})_{*-1}$. Then, using Lemma~\ref{lem:technicallemmafordiff}, we fill the top of the diagram, and we get $x \in (Y^{\phi(N+1)}_{\phi(N)})_*$, an $(M-N)$-cycle of $({}^\phi \mathcal{B})$ represented by $\hat{x}$.
\end{proof}

\section{Topological Hochschild homology of $ku$}
\label{chap:thhkucomplet}

Here we finish the computation of $\THH_*(ku)$. We will see that $u^{p-2}\sigma u$ is indeed a class of $\THH_{2p-1}(ku;ku/v_1)$ that can be compared to $\sigma v_1 \in \THH_*(\ell;H\Z_{(p)})$. We compute $\THH_*(ku;ku/v_1)$ in Section~\ref{section:thhkuv1} using a comparison between the Brun spectral sequences \eqref{ss:lZ} and \eqref{ss:uTB} and the truncated Bockstein spectral sequence \eqref{ss:uT} -- which has fewer classes and is easier to track.

The techniques we developed in Section~\ref{chap:truncated} can then be used to determine the $u$-Bockstein spectral sequence for $ku$, which is done in Section~\ref{section:thhku}. We can compare the $v_1$-Bockstein spectral sequences \eqref{ss:l} and \eqref{ss:v1}, and the Bockstein spectral sequence \eqref{ss:u} can be recovered from the truncated Bockstein spectral sequence \eqref{ss:uT} and the gathered Bockstein spectral sequence \eqref{ss:v1}.

Lastly, the extensions can be computed from the differentials and bidegree constraints of the Bockstein spectral sequence \eqref{ss:u}, thus determining $\THH_*(ku)$ as a $ku_*$-module.

Our $S$-algebra model for the quotient of $ku$ by $v_1$ will be
\begin{equation*}
  ku/v_1 = ku\wedge_\ell H\Z_{(p)}
\end{equation*}
which can also be assumed to be a $q$-cofibrant commutative $S$-algebra by Remark VII.6.8 of \cite{elmendorf1997rings}.

\subsection{Computation of $\THH_*(ku;ku/v_1)$}
\label{section:thhkuv1}

We will now compute $\THH_*(ku;ku/v_1)$ using both a Brun spectral sequence and a Bockstein spectral sequence. The following results allow us to compute the first page of the Brun spectral sequences.
\begin{lemma}
  \begin{enumerate}
  \item   $(ku/v_1\wedge_{ku} ku/v_1)_*\cong P_{p-1}(u)\otimes E(\sigma v_1)$ over $\Z_{(p)}$ with $|u|=2$ and $|\sigma v_1| = 2p-1$.
  \item   $V(0)_*(ku/v_1\wedge_{ku} ku/v_1)\cong P_{p-1}(u)\otimes E(\sigma v_1)$ over $\F_p$ with $|u|=2$ and $|\sigma v_1| = 2p-1$.
  \item $\THH_*(ku/v_1;H\Z_{(p)})\cong \THH_*(H\Z_{(p)})\otimes E(\sigma u) \otimes \Gamma(\varphi u)$ over $\Z_{(p)}$ with $|\sigma u|=3$ and $|\varphi u|=2p$.
  \item $\THH_*(ku/v_1;H\F_p)\cong V(0)_*\THH(H\Z_{(p)})\otimes E(\sigma u) \otimes \Gamma(\varphi u)$ over $\F_p$ with $|\sigma u|=3$ and $|\varphi u|=2p$.
  \end{enumerate}
\end{lemma}

\begin{proof}
  The Künneth spectral sequence computing $(ku/v_1\wedge_{ku} ku/v_1)_*$ has $E^2$-page $\operatorname{Tor}^{P(u)}_{*,*}(P_{p-1}(u),P_{p-1}(u)) = P_{p-1}(u)\otimes E(\sigma v_1)$ with $|u|=(0,2)$ and $|\sigma v_1|=(1,2p-2)$. For degree reasons, the spectral sequence collapses with no possible extensions, yielding claim 1. Claim 2 follows from the absence of $p$-torsion.

  We use the Brun spectral sequence to compute $\THH_*(ku/v_1;H\Z_{(p)})$ and $\THH_*(ku/v_1;H\F_p)$:
  \begin{equation*}
    \THH_*(H\Z_{(p)}; H(H\Z_{(p)}\wedge_{ku/v_1} H\Z_{(p)})_*) \Rightarrow \THH_*(ku/v_1;H\Z_{(p)})
  \end{equation*}
  \begin{equation*}
    \THH_*(H\F_p; H(H\F_p\wedge_{ku/v_1} H\F_p)_*) \Rightarrow \THH_*(ku/v_1;H\F_p) .
  \end{equation*}

  The Künneth spectral sequence computing $(H\Z_{(p)}\wedge_{ku/v_1} H\Z_{(p)})_*$ has $E^2$-page $\operatorname{Tor}^{P_{p-1}(u)}_{*,*}(\Z_{(p)},\Z_{(p)})\cong E(\sigma u)\otimes \Gamma(\varphi u)$ with $|\sigma u|=(1,2)$ and $|\varphi u|=(2,2p-2)$. The indecomposables are $\sigma u$ and the divided power $\gamma_{p^i}\varphi u$. For bidegree reasons, they cannot support non-zero differentials, so the spectral sequence collapses with no possible extensions, and we have $(H\Z_{(p)}\wedge_{ku/v_1}H\Z_{(p)})_*\cong E(\sigma u) \otimes \Gamma(\varphi u)$. A similar argument yields $(H\F_p\wedge_{ku/v_1} H\F_p)_*\cong  E(\sigma u) \otimes \Gamma(\varphi u)$, this time over $\F_p$.

  Returning to the Brun spectral sequences, when looking at the bidegrees modulo $2p$, we see that the indecomposables also cannot support non-zero differentials in both the integral and $V(0)$ case, so that the two spectral sequences collapse. The modulo $p$ $E^\infty$-page has exactly the right rank over $\F_p$ to fit into a long exact sequence induced by multiplication by $p$ for the integral $E^\infty$-page. Having an extension in the integral spectral sequence would then mean that there is a non-zero differential in the modulo $p$ one. We conclude that there is no extension in the integral spectral sequence, and claim 3 is proven. The modulo $p$ spectral sequence cannot have extensions, and we get claim 4.
\end{proof}

We can write the following two spectral sequences computing $\THH$ of $ku$ with coefficients in $ku/v_1$. The first one, \eqref{ss:uT}, is a truncated Bockstein spectral sequence:
\begin{gather}
  E^1_{*,*} = \THH_*(ku;H\Z_{(p)})\botimes P_{p-1}(u) \cong \THH_*(H\Z_{(p)})\otimes E(\sigma u) \botimes P_{p-1}(u) \notag \\
  \Rightarrow \THH_*(ku;ku/v_1) \tag{$u_T$}; \label{ss:uT} \\
\end{gather}
The second one, \eqref{ss:uTB}, is a Brun spectral sequence:
\begin{gather}
  E^2_{*,*} = \THH_*(ku/v_1;H(ku/v_1\wedge_{ku}ku/v_1)_*) \notag \\
  \cong\THH_*(H\Z_{(p)}) \otimes E(\sigma u)\otimes \Gamma(\varphi u) \botimes E(\sigma v_1) \otimes P_{p-1}(u) \notag \\
  \Rightarrow \THH_*(ku;ku/v_1) \tag{$u_{TB}$}. \label{ss:uTB} \\
\end{gather}
The bidegrees are:
\begin{equation*}
  \begin{aligned}
    & |\sigma u| = (3,0) \\
    & |\varphi u| = (2p,0) \\
    & |\mu_{kp}| = (2kp-1,0) \mbox{, $k\geq 1$ the generators of $\THH_*(H\Z_{(p)})$} \\
    & |u| = (0,2) \\
    & |\sigma v_1| = (0,2p-1).
  \end{aligned}
\end{equation*}

For the following lemma, we will briefly use the non-truncated $u$-Bockstein spectral sequence ($u$) computing $\THH_*(ku)$ that we will study in the next section. It links the class $\sigma v_1$ of $\THH_*(\ell)$ to a class of $\THH_*(ku)$.
\begin{lemma}\label{prop:imagesigmav1}
  The map $\THH_*(\ell)\rightarrow \THH_*(ku)$ sends $\sigma v_1$ to a non-zero class represented up to a unit by $u^{p-2}\sigma u$ in the Bockstein spectral sequence \eqref{ss:u} computing $\THH_*(ku)$.

  The class $\mu_p \in \THH_*(ku; H\Z_{(p)})$ is an infinite cycle in \eqref{ss:u} and there is a relation $p \cdot \mu_p = u^{p-2} \sigma u$ in $\THH_*(ku)$.
\end{lemma}
\begin{proof}
  For degree reasons, the classes $u^{p-2} \sigma u, \mu_p \in \THH_*(ku; H\Z_{(p)})$ are infinite cycles in the Bockstein spectral sequence \eqref{ss:u}, and they are the only generators in degree $2p -1$. There is a commutative diagram
  \begin{center}
    \begin{tikzcd}
      \THH_*(\ell) \rar \dar & \THH_*(ku) \dar \\
      \THH_*(\ell; H\Z_{(p)}) \rar & \THH_*(ku; H\Z_{(p)})
    \end{tikzcd}
  \end{center}
  where we can push $\mu_p$. The image of the element $\mu_p$ of $\THH_*(\ell; H\Z_{(p)})$ is $\mu_p$ in $\THH_*(ku; H\Z_{(p)})$ from comparing the Brun spectral sequences computing $\THH_*(\ell; H\Z_{(p)})$ and $\THH_*(ku; H\Z_{(p)})$. Thus, the image of $\mu_p$ in $\THH_*(ku)$ must be represented by $\mu_p \in \THH_*(ku; H\Z_{(p)})$ in the Bockstein spectral sequence \eqref{ss:u}. Moreover, from the description of $V(0)_*\THH(ku)$ found in Theorem 7.9 of \cite{ausoni2005topological}, we see that $V(0)_{2p -1}\THH(ku) = \F_{p}$ so that $\THH_{2p -1}(ku)$ must be cyclic. The only possibility is that there is an extension $p \cdot \mu_p = u^{p -2} \sigma u$ up to a unit, and since we have $p \cdot \mu_p = \sigma v_1$ in $\THH_*(\ell)$, the image of $\sigma v_1$ in $\THH_*(ku)$ is $u^{p-2} \sigma u$ up to a unit.
\end{proof}

Since $\ell/v_1$ is just $H\Z_{(p)}$, we have a morphism between the Brun spectral sequences $\eqref{ss:lZ}\rightarrow \eqref{ss:uTB}$ induced by the inclusion of the summand $i:\ell\rightarrow ku$. This allows us to prove:
\begin{proposition}\label{prop:diffin6}
  In \eqref{ss:uTB}, there are differentials
  \begin{equation*}
    d^{2p-3}(\gamma_k\varphi u)=u^{p-2}\sigma u\gamma_{k-1}\varphi u
  \end{equation*}
  up to a unit for all $k\geq 1$.
\end{proposition}
\begin{proof}
  In the following commutative diagram:
  \begin{center}
    \begin{tikzcd}
      \THH(\ell) \rar["i"] \dar["f"] & \THH(ku) \dar["f"] \\
      \THH(\ell;H\Z_{(p)}) \rar["i"] & \THH(ku;ku/v_1)
    \end{tikzcd}
  \end{center}
  Up to units we have, using Lemma~\ref{prop:imagesigmav1}
  \begin{equation*}
    f(i(\sigma v_1)) = f(u^{p-2}\sigma u) = u^{p-2} f(\sigma u) = i(f(\sigma v_1)) = i(\sigma v_1) .
  \end{equation*}

  In order for this to be possible, there must be an extension $u\cdot u^{p-3}\sigma u = \sigma v_1$ in \eqref{ss:uTB}, and it must be that $u^{p-2}\sigma u$ is either a boundary or not an infinite cycle. Since it is an infinite cycle for degree reasons, it is a boundary. The only class in degree $2p$ is $\varphi u$, so up to a unit there is a differential $d^{2p-3}(\varphi u)=u^{p-2}\sigma u$ in \eqref{ss:uTB}.

  In the divided power algebra $\Gamma(\varphi u)$, $\varphi u\,\gamma_{k-1} \varphi u = k\,\gamma_k\varphi u$. We can then prove our formula by induction on $k$, using the facts that
  \begin{equation*}
    k\,d(\gamma_k\varphi u) = d(\varphi u) \gamma_{k-1}\varphi u + \varphi u\,d(\gamma_{k-1}\varphi u)
  \end{equation*}
  and that $\Z_{(p)}$ is an integral domain.
\end{proof}

We can now get a description of all the differentials in the truncated Bockstein spectral sequence \eqref{ss:uT}:
\begin{proposition}\label{prop:diffin5}
  In the spectral sequence \eqref{ss:uT}, the differentials are given by the formula:
  \begin{equation*}
    d^{p-2}(\mu_{(k+1)p})=p^{\nu(k)}u^{p-2}\sigma u\,\mu_{kp},\quad k\geq 1
  \end{equation*}
  up to a unit, where $\nu$ is the $p$-adic valuation.
\end{proposition}
\begin{proof}
  The differentials given are the only possible ones in \eqref{ss:uT} for bidegree reasons; we only need to prove that they are indeed non-zero. We now know enough about \eqref{ss:uTB} to do so.

  By looking at the degrees modulo $2p$, we can list the classes of total degree $2kp - 1$ in $E^1$ of \eqref{ss:uTB}:
  \begin{equation*}
    \mu_{kp}, \quad \gamma_{k-1}\varphi u\,\sigma v_1, \quad \gamma_{k-1}\varphi u\,u^{p-2}\sigma u,  \quad \gamma_i\varphi u\,\mu_{(k-i)p},\; 1\leq i < k .
  \end{equation*}

  We know the following differentials in \eqref{ss:uTB}:
  \begin{equation*}
    d^{2p-3}(\gamma_i\varphi u)=u^{p-2}\sigma u\gamma_{i-1}\varphi u
  \end{equation*}
  for $i\geq 1$ from Proposition~\ref{prop:diffin6};
  \begin{equation*}
    d^{2p}(\mu_{(i+1)p})=p^{\nu(i)}\sigma v_1\mu_{ip}
  \end{equation*}
  from the map $\eqref{ss:lZ}\rightarrow \eqref{ss:uTB}$ and Proposition~\ref{prop:diffin4}.

  To complete the multiplicative description, we also note that $\sigma v_1$ is an infinite cycle for bidegree reasons, and that all the degreewise possible values for $d^{2p-3}(\mu_{(k-i)p})$ result in a non-zero $d^{2p-3}(\gamma_i\varphi u\,\mu_{(k-i)p})$.

  From this description, after $d^{2p}$ the only generator left in $E^{2p+1}$ in total degree $2kp -1$ is $p\mu_{kp}$, so that $\THH_{2kp-1}(ku;ku/v_1)$ is isomorphic to $\Z/p^{\nu(k)}\Z$. This proves our claim about \eqref{ss:uT}.
\end{proof}

We will now describe $\THH_*(ku;ku/v_1)$; once again, we will use $v_0$ to denote multiplication by $p$ in the $E^\infty$-page of the spectral sequence \eqref{ss:uT}, and $p\cdot$ to denote the multiplication in the target group.
\begin{proposition}
  $\THH_*(ku;ku/v_1)$ is generated as a $\Z_{(p)}[u]/(u^{p-1})$-module by
  \begin{equation*}
    \begin{gathered}
      1,\, \sigma u,\, \mu_p \\
      v_0\mu_{kp},\, u\mu_{kp},\, k\geq 2 \\
      \sigma u\mu_{kp},\, k\geq 1 
    \end{gathered}
  \end{equation*}
  with the relations:
  \begin{equation*}
    \begin{aligned}
      u^{p-2}\cdot\sigma u & = p\cdot \mu_p \\
      u\cdot v_0\mu_{kp} & = p\cdot u\mu_{kp},\, k\geq 2 \\
      p^{\nu(k)+1}\cdot u\mu_{kp} & = 0 ,\, k\geq 2\\
      u^{p-3}\cdot u\mu_{kp} & = 0 ,\, k\geq 2\\
      p^{\nu(k)+1}\cdot \sigma u\mu_{kp} & = 0 ,\, k\geq 2 \\
      p^{\nu(k)}\cdot v_0\mu_{kp} & = 0,\, k\geq 2\\
      p^{\nu(k)}u^{p-2}\cdot \sigma u\mu_{kp} & = 0,\, k\geq 2.
    \end{aligned}
  \end{equation*}
\end{proposition}
\begin{proof}
  Except for the extension $p\cdot \mu_p = u^{p-2}\sigma u$ this is the $E^\infty$-page of \eqref{ss:uT}. This extension is present in \eqref{ss:lZ}, and since from Lemma~\ref{prop:imagesigmav1} the map $i:\THH_*(\ell;H\Z_{(p)})\rightarrow \THH_*(ku;ku/v_1)$ is such that $i(\sigma v_1) = u^{p-2}\sigma u$, and $i(\mu_p)=\mu_p$, it must be that $u^{p-2}\sigma u$ is also divisible by $p$ in $\THH_{2p-1}(ku;ku/v_1)$. The only possible extension is with $\mu_p$, so we get our formula up to a unit.
\end{proof}

Without the module structure, writing all the classes, we obtain:
\begin{equation*}
  \begin{aligned}
    & \Z_{(p)}\{1,\, u,\, \dots,\, u^{p-2},\, \sigma u,\,  u\sigma u,\, \dots,\, u^{p-2}\sigma u,\,  \mu_p\} \\
    & \oplus\bigoplus_{k\geq 1} \faktor{\Z}{p^{\nu(k)+1}}\{u\mu_{kp},\, u^2\mu_{kp},\, \dots,\, u^{p-2}\mu_{kp}\} \\
    & \oplus\bigoplus_{k\geq 1} \faktor{\Z}{p^{\nu(k)+1}}\{\sigma u\mu_{kp},\, u\mu_{kp},\, \dots,\, u^{p-3}\mu_{kp}\} \\
    & \oplus\bigoplus_{k\geq 2} \faktor{\Z}{p^{\nu(k)}}\{v_0\mu_{kp},\, u^{p-2}\sigma u\mu_{kp}\}
  \end{aligned}
\end{equation*}
with relations $u^{p-2}\sigma u = p\cdot \mu_p$ and $u\cdot v_0\mu_{kp}=p\cdot u\mu_{kp}$.

\subsection{Computation of the Bockstein spectral se\-quen\-ce for $\THH_*(ku)$}
\label{section:thhku}

We know enough about these first three spectral sequences to compute the fourth:
\begin{align}
    & \THH_*(\ell;H\Z_{(p)})\botimes P(v_1) && \Rightarrow \THH_*(\ell) \tag{\ref{ss:l}} \\
    & \THH_*(ku;H\Z_{(p)})\botimes P_{p-1}(u) && \Rightarrow \THH_*(ku;ku/v_1) \tag{\ref{ss:uT}} \\
    & \THH(ku;ku/v_1)\botimes P(v_1) && \Rightarrow \THH_*(ku) \tag{$v_1$} \label{ss:v1} \\
    & \THH_*(ku;H\Z_{(p)})\botimes P(u) && \Rightarrow \THH_*(ku). \tag{$u$} \label{ss:u}
\end{align}

The map $\THH(\ell;H\Z_{(p)})\rightarrow \THH(ku;ku/v_1)$ induces a morphism of spectral sequences $\eqref{ss:l}\rightarrow \eqref{ss:v1}$, which determines some differentials in \eqref{ss:v1}. These differentials, together with the ones computed in the previous section in \eqref{ss:uT} and with Theorems~\ref{prop:totalandsmall}, \ref{prop:totaltotruncated}, \ref{prop:truncatedtototal} and \ref{prop:nulldiffs} relating spectral sequences and their truncations, yield a description of the differentials in \eqref{ss:u}.

\begin{theorem}\label{prop:diffinu}
  The differentials in \eqref{ss:u} are given by the formula:
  \begin{equation*}
    d^{p^{n+1}-2}(p^n\mu_{(k+1)p^{n+1}}) = ku^{p^{n+1}-2}\sigma u \mu_{kp^{n+1}},\; k\geq 0,\; n\geq 0
  \end{equation*}
  up to a unit and linearity with respect to multiplication by $u$.
\end{theorem}
\begin{proof}
  Here we make good use of our results on truncated spectral sequences.

  First, the differentials in \eqref{ss:uT} from Proposition~\ref{prop:diffin5} are lifted to \eqref{ss:u} using Theorem~\ref{prop:totalandsmall}, that is in \eqref{ss:u} there are differentials:
  \begin{equation}
     \quad d^{p-2}(\mu_{(k+1)p})=p^{\nu(k)}u^{p-2}\sigma u\,\mu_{kp},\quad k\geq 1 \label{eq:somediffinu}
  \end{equation}
  and the corresponding differentials obtained by multiplying the source and the target by any power of $u$. These are the only differentials $d^r$ with $1\leq r\leq p-2$ in \eqref{ss:u} since these are the only differentials in \eqref{ss:uT}, again using Theorem~\ref{prop:totalandsmall}.

  We will now use Theorem~\ref{prop:totaltotruncated} and Theorem~\ref{prop:truncatedtototal}. With regard to Theorem~\ref{prop:truncatedtototal}, in our current computation, a statement stronger than the general case can be made. The general case would say that a differential $d(x)=y$ in \eqref{ss:v1} would result in the existence of an element $x'$ such that $d(x')=y$ in \eqref{ss:v1}, and such that this differential can be lifted to one in \eqref{ss:u}; but in \eqref{ss:v1}, each generator is alone in its bidegree, so that necessarily $x=x'$. Therefore, each differential $d(x)=y$ in \eqref{ss:v1} can really be lifted to a differential $d(x)=y$ in \eqref{ss:u}.

  Using Theorem~\ref{prop:totaltotruncated}, the differentials of formula \eqref{eq:somediffinu} result in \eqref{ss:v1} in
  \begin{equation*}
    d^1(u^i\mu_{(k+1)p})=p^{\nu(k)}v_1u^{i-1}\sigma u\,\mu_{kp},\quad k\geq 1,\quad 1\leq i\leq p-2
  \end{equation*}
  and the corresponding differentials obtained by multiplying the source and the target by any power of $v_1$. These are the only differentials $d^1$ in \eqref{ss:v1} since having more differentials would result in more differentials $d^r$ in \eqref{ss:u} with $1\leq r \leq p-2$. This gives the $E^2$-page of \eqref{ss:v1}:
  \begin{equation*}
    \begin{aligned}
      \eqref{ss:v1} :\; E^2 \cong & \Z_{(p)}\{1,\, u,\, \dots,\, u^{p-2},\, \sigma u,\,  u\sigma u,\, \dots,\, u^{p-3}\sigma u,\,  \mu_p\}\otimes P(v_1) \\
                         &  \oplus\bigoplus_{k\geq 2} \faktor{\Z}{p^{\nu(k)}}\{v_0\mu_{kp}\}\otimes P_{p-1}(u)\otimes P(v_1) \\
                         & \oplus\bigoplus_{k\geq 1} \faktor{\Z}{p^{\nu(k)+1}}\{\sigma u\mu_{kp},\, u\,\sigma u\mu_{kp},\, \dots,\, u^{p-3}\sigma u\mu_{kp}\} \\
                         & \oplus\bigoplus_{k\geq 1} \faktor{\Z}{p^{\nu(k)}}\{u^{p-2}\sigma u\mu_{kp},\, v_1\sigma u\mu_{kp},\, u\,v_1\sigma u\mu_{kp},\, \dots\}.
    \end{aligned}
  \end{equation*}

  We have written all the generators $v_0\mu_{kp}$ with $v_0$ because we will now account for the differentials in \eqref{ss:l} of Theorem~\ref{prop:diffinell}:
  \begin{equation*}
    d^{p^n+\dots +p}(p^{n-1}\cdot v_0\mu_{kp^{n+1}}) = (k-1) v_1^{p^n+\dots+p}\sigma v_1\mu_{(k-1)p^{n+1}},\; k\geq 1,\; n\geq 1.
  \end{equation*}
  Because of the morphism of spectral sequences $(\ell) \rightarrow (v_1)$, that formula is also true in \eqref{ss:v1}, and from Theorem~\ref{prop:truncatedtototal} we deduce the formula in \eqref{ss:u} that was claimed (which also encompasses formula \eqref{eq:somediffinu}).

  It remains to prove that the classes $\sigma u \mu_{kp},\, k\geq 1$ are infinite cycles in \eqref{ss:u}. The classes $u^{p-2}\sigma u \mu_{kp^2},\,k\geq 1$ are in the image of $\eqref{ss:l}\rightarrow \eqref{ss:v1}$ and so are infinite cycles in \eqref{ss:v1}, thus also in \eqref{ss:u} by Theorem~\ref{prop:nulldiffs}. Since in \eqref{ss:u} the only $u^{p-2}$-torsion is in even degree, it must be that $\sigma u \mu_{kp^2}$ are infinite cycles in \eqref{ss:u}. The remaining classes to check are the $\sigma u\mu_{kp}$ with $p$ not dividing $k$. Once again we know that these classes support no differentials of height up to $u^{p-2}$, and are of $u^{p-2}$-torsion after $d^{p-2}$ by formula \eqref{eq:somediffinu}. If some $\sigma u\mu_{kp}$ supports a non-zero differential the target must be $p^{\nu(k-i)}u^{ip+1}\mu_{(k-i)p},\, 1\leq i \leq k-1$ for degree reasons, and that target must be of $u^{p-2}$-torsion, that is to say some
  \begin{equation*}
    \begin{gathered}
      p^{\nu(k-i)}u^{ip+2}\mu_{(k-i)p}, \\
      p^{\nu(k-i)}u^{ip+3}\mu_{(k-i)p}, \\
      \vdots \\
      p^{\nu(k-i)}u^{ip+p-1}\mu_{(k-i)p}
    \end{gathered}
  \end{equation*}
  is already the target of a differential. But the only possible differentials still not accounted for are the ones targeting $p^{\nu(k-i)}u^{ip+1}\mu_{(k-i)p},\, 1\leq i \leq k-1$, and these are of height $u^h$ with $h$ reducing to 1 modulo $p$.
\end{proof}

We can change our generators so that the differentials are not given up to a unit but are given exactly.
\begin{proposition} \label{prop:diffexactes}
  We can change the generators $\mu_N$ and $\sigma u\mu_N$ of $\THH_*(ku;H\Z_{(p)})$ with a multiplication by a unit so that the differentials in \eqref{ss:u} are given by the formula:
  \begin{equation*}
    d^{p^{n+1}-2}(p^n\mu_{(k+1)p^{n+1}}) = p^{\nu(k)}u^{p^{n+1}-2}\sigma u \mu_{kp^{n+1}},\; k\geq 0,\; n\geq 0
  \end{equation*}
\end{proposition}
\begin{proof}
  Note that we have chosen $p^{\nu(k)}$ instead of $k$, but these are the same up to a unit. We could have written the same statement with $k$.

  The differentials are making the $\mu_N$ and $\sigma u\mu_{N}$ interact, and once we have chosen a specific unit for one of them, we have to use the same unit for all their multiplication by powers of $p$. Consider the simple, unoriented graph $\mathcal{G}$ whose vertices are the classes $\mu_N$ and $\sigma u \mu_N$ for any $N \geq 0$ divisible by $p$, with an edge $\mu_N \text{---}\sigma u\mu_{N'}$ whenever there is a differential
  \begin{equation*}
    d(p^i\mu_N) = p^ju^\bullet\sigma u\mu_{N'}
  \end{equation*}
  for any $i$ and $j$, up to a unit, in the spectral sequence. The graph $\mathcal{G}$ is bipartite, with classes given by the presence or absence of $\sigma u$ in the vertex name. If we prove that $\mathcal{G}$ is acyclic, then we have proven our statement. Indeed, $\mathcal{G}$ is then a collection of trees; we can choose an arbitrary root in each connected component of $\mathcal{G}$; starting from the roots, we can change each generation of the trees by a unit to get the formula we want.

  We will reason about the $p$-adic valuation of $N$ and $N'$, denoted $\nu(N)$ and $\nu(N')$. There is an edge $\mu_N \text{---}\sigma u\mu_{N'}$ in $\mathcal{G}$ if and only if there exists $(k,n) \in \N^2$ such that
  \begin{equation*}
    N = (k+1)p^{n+1} \quad N' = kp^{n+1}.
  \end{equation*}

  Fix $N$, $k$ and $n$ such that $N = (k+1)p^{n+1}$, let $N' = kp^{n+1}$. Then
  \begin{equation*}
    \nu(N') \geq \nu(N) \iff n+1 = \nu(N)
  \end{equation*}
  so that there can only be one edge $\mu_N\text{---}\sigma u\mu_{N'}$ such that $\nu(N') \geq \nu(N)$.

   Fix $N'$, $k$ and $n$ such that $N' = kp^{n+1}$, let $N = (k+1)p^{n+1}$. Then
  \begin{equation*}
    \nu(N) \geq \nu(N') \iff n+1 = \nu(N')
  \end{equation*}
  so that there can only be one edge $\sigma u\mu_{N'}\text{---}\mu_{N}$ such that $\nu(N) \geq \nu(N')$.

  Combining both cases, we see that from any vertex of $\mathcal{G}$, there is exactly one edge such that the $p$-adic valuation is non-decreasing. If there were a non-trivial cycle in $\mathcal{G}$, we could extract an irreducible cycle that goes through any edge at most once. Such a cycle would be confined to vertices whose $p$-adic valuation is constant; otherwise each edge would strictly decrease the valuation. But any vertex has at most one neighbor whose $p$-adic valuation is the same; thus $\mathcal{G}$ is acyclic.
\end{proof}

\subsection{Computing the extensions and a presentation of $\THH_*(ku)$}

We first compute the extensions in the torsion-free part of the spectral sequence, from the knowledge that the $p$-torsion and the $u$-torsion must be the same in $\THH_*(ku)$.
\begin{proposition}
  The torsion-free part of $\THH_*(ku)$ is a quotient of
  \begin{equation*}
    P(u)\otimes \Z_{(p)}\{1,\,\sigma u,\, \mu_p,\,v_0\mu_{p^2},\,v_0^2\mu_{p^3},\,\dots \}
  \end{equation*}
  with relations
  \begin{equation*}
    p\cdot \mu_p = u^{p-2}\sigma u
  \end{equation*}
  \begin{equation*}
    p\cdot v_0^n\mu_{p^{n+1}} = u^{p^{n+1}-p^n}v_0^{n-1}\mu_{p^n},\, n\geq 1.
  \end{equation*}
\end{proposition}
\begin{proof}
  This proof continues the reasoning of the proof of Lemma~\ref{prop:imagesigmav1}. The first extension is already proved in that lemma.
  We have a commutative diagram
  \begin{center}
    \begin{tikzcd}
      \THH_*(\ell) \rar \dar & \THH_*(ku) \dar \\
      \THH_*(\ell; H\Z_{(p)}) \rar & \THH_*(ku; H\Z_{(p)})
    \end{tikzcd}
  \end{center}
  where we have named classes $v_0^n \mu_{p^{n+1}}$ in $\THH_*(\ell)$, $\THH_*(\ell; H\Z_{(p)})$ and $\THH_*(ku; H\Z_{(p)})$ that are their respective images. Since there is already a relation
  \begin{equation*}
    p \cdot v_0^n \mu_{p^{n +1}} = v_1^{p^n}v_0^{n-1} \mu_{p^n}
  \end{equation*}
  in $\THH_*(\ell)$ and since $v_1 = u^{p-1}$ in $\THH_*(ku)$, we can see that the image of $v_0^n \mu_{p^{n+1}}$ in $\THH_*(ku)$ must be a class represented by $v_0^n \mu_{p^{n+1}} \in \THH_*(ku; H\Z_{(p)})$ in the spectral sequence \eqref{ss:u}, and that there is a relation
  \begin{equation*}
    p\cdot v_0^n\mu_{p^{n+1}} = u^{p^{n+1}-p^n}v_0^{n-1}\mu_{p^n}
  \end{equation*}
  in $\THH_*(ku)$.
\end{proof}

We now give a presentation of the torsion. Before computing the extensions, we need to ensure that the lifts we choose of the classes of the spectral sequence have the correct properties.

\begin{lemma} \label{lem:liftinssu}
  For any $n\geq 1$, $0\leq h < n$ and $a\geq 1$, $a$ not divisible by $p$, the infinite cycle
  \begin{equation*}
    p^h\sigma u \mu_{ap^n}
  \end{equation*}
  of the spectral sequence \eqref{ss:u} lifts to an element
  \begin{equation*}
    v_0^h\sigma u\mu_{ap^n}
  \end{equation*}
  in $\THH_*(ku)$ such that
  \begin{equation} \label{eq:thhkuliftuzero}
    u^{p^{n-h}-2}\cdot v_0^h\sigma u \mu_{ap^n} = 0 
  \end{equation}
  and if $a = bp + p - 1$ for some $b \geq 1$,
  \begin{equation} \label{eq:thhkuliftpzero}
    p\cdot u^{p^n-3}\cdot \sigma u\mu_{(bp + p -1)p^n} = u^{p^{n+1} - 3} v_0^{\nu(b)}\sigma u\mu_{bp^{n+1}}
  \end{equation}
  otherwise we have
  \begin{equation} \label{eq:thhkuliftpzeronotdiv}
    p\cdot u^{p^n-3}\cdot v_0^h\sigma u\mu_{ap^n} = 0.
  \end{equation}
  Moreover, in both cases, $p \cdot v_0^h\sigma u\mu_{ap^n}$ and $v_0^{h + 1}\sigma u \mu_{ap^n}$ differ only by an element divisible by $u$.
\end{lemma}

\begin{proof}
  We see these relations in the exact couple defining the spectral sequence \eqref{ss:u}. Since there are differentials
  \begin{equation*}
    d^{p^{n+1}-2}(p^n\mu_{(k+1)p^{n+1}}) = p^{\nu(k)}u^{p^{n+1}-2}\sigma u \mu_{kp^{n+1}},\; k\geq 0,\; n\geq 0
  \end{equation*}
  we can populate the diagram
    \begin{equation*}
    \begin{tikzcd}[column sep = small]
      A \ar[rr, "u"] & & A \ar[rr, "u"] \ar[dl, "j"] & & ... \ar[rr, "u"] & & A \ar[rr, "u"] & & A \ar[dl, "j"] \\
      & B \ar[ul, "\partial"] & & &   & & & B \ar[ul, "\partial"] & \\
      & & \alpha \ar[rr, mapsto] \ar[dl, mapsto] & & ... \ar[rr, mapsto] & & \beta \ar[rr, mapsto] & & 0 \\
      & p^{\nu(k)}\sigma u\mu_{kp^{n+1}} & & &   & & & p^n\mu_{(k+1)p^{n+1}} \ar[ul, mapsto] &
    \end{tikzcd}
  \end{equation*}
  where $A = \THH_*(ku)$ and $B=\THH_*(ku;H\Z_{(p)})$. Here we set $\beta = \partial(p^n\mu_{(k+1)p^{n+1}})$. The existence of the differential ensures that $\beta$ lifts $p^{n+1}-3$ times through the multiplication by $u$ map to an element $\alpha\in\THH_*(ku)$, and that $j(\alpha) = p^{\nu(k)}\sigma u \mu_{kp^{n+1}}$, that is $\alpha$ is represented by the infinite cycle $p^{\nu(k)}\sigma u \mu_{kp^{n+1}}$ of the spectral sequence. We then name that element of $\THH_*(ku)$
  \begin{equation*}
    \alpha = v_0^{\nu(k)}\sigma u\mu_{kp^{n+1}}
  \end{equation*}
 which implies that
 \begin{equation*}
   \begin{aligned}
     u^{p^{n+1}-3}v_0^{\nu(k)}\sigma u\mu_{kp^{n+1}} & = \beta \\
     u^{p^{n+1}-2}v_0^{\nu(k)}\sigma u\mu_{kp^{n+1}} & = 0
   \end{aligned}
 \end{equation*}
 which is another way to write \eqref{eq:thhkuliftuzero}.
 
 Moreover, $\partial$ is a map of $\Z_{(p)}$-modules. Thus,
 \begin{equation*}
   \partial(p^{n+1}\mu_{(k+1)p^{n+1}}) = p\cdot \partial(p^n\mu_{(k+1)p^{n+1}}) \in \THH_*(ku).
 \end{equation*}
 When $k+1$ is divisible by $p$, write $k= bp + p - 1$ so that $k+1 = (b+1)p$, the left hand side is already named
 \begin{equation*}
  \partial (p^{n+1}\mu_{(b+1)p^{n+2}}) = u^{p^{n+2}-3}v_0^{\nu(b)}\sigma u \mu_{bp^{n+2}}
 \end{equation*}
 and the right hand side is
 \begin{equation*}
   \begin{aligned}
     p\cdot \partial(p^n\mu_{(b+1)p^{n+2}})
     & = p\cdot u^{p^{n+1} - 3}v_0^{\nu((b+1)p - 1)} \sigma u\mu_{((b+1)p - 1)p^{n+1}} \\
     & = p\cdot u^{p^{n+1} - 3}\sigma u\mu_{(bp + p - 1)p^{n+1}}
   \end{aligned}
 \end{equation*}
 which yields \eqref{eq:thhkuliftpzero} written with $n\geq 0$ instead of $n \geq 1$.

 If $k+1$ is not divisible by $p$, we have $p^{n+1}\mu_{(k+1)p^{n+1}} = 0$ in $\THH_*(ku;H\Z_{(p)})$ so
 \begin{equation*}
   p\cdot \partial(p^n\mu_{(k+1)p^{n+1}}) = p\cdot u^{p^{n+1} - 3}v_0^{\nu(k)} \sigma u\mu_{kp^{n+1}} = 0
 \end{equation*}
 which is \eqref{eq:thhkuliftpzeronotdiv}.

 In both cases, we have
 \begin{equation*}
   j(p\cdot v_0^h\sigma u\mu_{ap^n}) = p^{h+1}\sigma u \mu_{ap^n} = j(v_0^{h+1}\sigma u \mu_{ap^n})
 \end{equation*}
 that is, $p\cdot v_0^h\sigma u\mu_{ap^n}$ and $v_0^{h+1}\sigma u\mu_{ap^n}$ are represented by the same class in the spectral sequence. By exactness of the diagram, their difference must be a multiple of $u$.
\end{proof}

Having constructed lifts of all the torsion infinite cycles, we can recover the torsion extensions.
\begin{proposition}
  The torsion $ku_*$-sub-module of $\THH_*(ku)$ is presented by the classes
  \begin{equation*}
    v_0^h\sigma u\mu_{ap^n}
  \end{equation*}
  in degree $2ap^n+2$ where $h$, $a$ and $n$ are integers such that $h\geq 0$, $n\geq 1$, $a\geq 1$ and $p$ does not divide $a$, together with the relations:
  \begin{enumerate}
  \item $v_0^{h} \sigma u \mu_{ap^n} = 0$ for any $a \geq 1$ not divisible by $p$, $n \geq 1$ and $h \geq n$,
  \item $u^{p^{n-h}-2}\cdot v_0^h\sigma u \mu_{ap^n} = 0$ for any $a \geq 1$ not divisible by $p$,  $n \geq 1$ and $0 \leq h \leq n - 1$,
  \item $p\cdot \sigma u \mu_{(bp+p-1)p^n} = v_0\sigma u\mu_{(bp+p-1)p^n}+u^{p^{n+1}-p^n}v_0^{\nu(b)}\sigma u \mu_{bp^{n+1}}$ for any $b\geq 1$ and any $n \geq 1$,
  \item  $p\cdot v_0^h\sigma u \mu_{ap^n} = v_0^{h+1}\sigma u \mu_{ap^n}$ for any $a \geq 1$ not divisible by $p$, $n \geq 1$ and $h\geq 1$ or $h=0$ not in case 3.
  \end{enumerate}
\end{proposition}

\begin{proof}
  The lifts of Lemma~\ref{lem:liftinssu} have the claimed properties. Claim 1 follows from the order with respect to multiplication by $p$ in the $E^\infty$-page of the spectral sequence. Claim 2 is proven in the lemma. We need to check 3 and 4. For any integer $m$ and any $h \geq 0$, the lemma states that in $\THH_*(ku)$, $p\cdot v_0^h\sigma u\mu_m$ must be equal to $v_0^{h+1}\sigma u\mu_m + u\alpha$ for some $\alpha\in\THH_{*-2}(ku)$. 
  
  We consider claim 3. Let $b\geq 1$ and $n\geq 1$. Since
  \begin{equation*}
    p\cdot u^{p^n-3}\cdot \sigma u\mu_{(bp + p -1)p^n} = u^{p^{n+1} - 3} v_0^{\nu(b)}\sigma u\mu_{bp^{n+1}}
  \end{equation*}
  we must have
  \begin{equation*}
    p\cdot \sigma u\mu_{(bp + p -1)p^n} = v_0\sigma u\mu_{(bp + p -1)p^n} + u^{p^{n+1} - p^n} v_0^{\nu(b)}\sigma u\mu_{bp^{n+1}} + u\alpha
  \end{equation*}
  with $\alpha$ such that $u^{p^n-2} \alpha = 0$.

  We now consider claim 4. Let $a \geq 1$ not divisible by $p$, $h\geq 0$ and $n\geq 1$. As before, we must have
  \begin{equation*}
    p\cdot v_0^h\sigma u\mu_{ap^n} = v_0^{h+1}\sigma u \mu_{ap^n} + u\alpha
  \end{equation*}
  with $u^{p^{n-h}-2}\alpha = 0$.

  We will prove that in both cases, $\alpha$ must be zero. Since Lemma~\ref{lem:liftinssu} lifts all the torsion classes of the spectral sequence, $\alpha$ must be written with the elements we have lifted. Let $a \geq 1$ not divisible by $p$, $h\geq 0$ and $n\geq 1$ be fixed (for claim 3, we can set $a = bp + p -1$ and $h =0$). To write $\alpha$, we need to find $c\geq 1$ not divisible by $p$, $m\geq 1$ and $k\geq 0$ such that
  \begin{equation}\label{eq:forsTinTHHku}
    |v_0^k\sigma u\mu_{cp^m}| < |v_0^h\sigma u\mu_{ap^n}| < |u^{p^{m-k}-2}v_0^k\sigma u\mu_{cp^m}| \leq |u^{p^{n-h}-2}v_0^h\sigma u\mu_{ap^n}|
  \end{equation}
  in order for the $u$-tower above $\alpha$ to end before the $u$-tower above $v_0^h\sigma u \mu_{ap^n}$.

  Computing the degree, we get
  \begin{equation*}
    \begin{gathered}
      2cp^m + 2 < 2ap^n + 2 < 2cp^m + 2 + 2(p^{m - k} - 2) < 2ap^n + 2 + 2(p^{n-h} - 2) \\
      \iff 0 < ap^n - cp^m < p^{m - k} - 2 < ap^n - cp^m + p^{n - h} - 2.
    \end{gathered}
  \end{equation*}

  If $m \leq n$, dividing by $p^m$, we get
  \begin{equation*}
    0 < ap^{n - m} - c < p^{-k} - 2p^{-m} < ap^{n-m}-c +p^{n-m-h} - 2p^{-m}
  \end{equation*}
  which is impossible since $ap^{n-m} -c$ is an integer.

  If $m > n$, dividing by $p^n$, we get
  \begin{equation*}
    \begin{gathered}
      0 < a - cp^{m -n} < p^{m - n - k} - 2p^{-n} < a - cp^{m - n} + p^{-h} - 2p^{-n} \\
      \iff 2p^{-n} < a - cp^{m -n} + 2p^{-n}< p^{m - n - k} < a - cp^{m - n} + p^{-h}.
    \end{gathered}
  \end{equation*}
  so that $p^{m - n -k} \geq 1$ must be an integer since $a - cp^{m - n}$ is also an integer. But there can be no integer in the open interval $(a - cp^{m-n} + 2p^{-n}, a - cp^{m-n} + p^{-h})$.

  Therefore, $\alpha$ must be zero. For degree reasons, the only possible extensions are the ones that we claim. As we know the relations at the top of the $u$-towers, we know these extensions must happen.
\end{proof}

From the two previous results, we can give a presentation of $\THH_*(ku)$ as a $ku_*$-module.
\begin{theorem} \label{prop:thhku}
  $\THH_*(ku)$ is a quotient of the $\Z_{(p)}[u]$-module
  \begin{multline}
   \Z_{(p)}[u]\{1,\,\sigma u,\, v_0^n\mu_{p^{n+1}},\,n\geq 0\} \\
    \oplus \Z_{(p)}[u]\{v_0^h\sigma u\mu_{ap^n},\, n\geq 1,\, a\geq 1,\, \mbox{$a$ not divisible by $p$},\, h\geq 0\}
  \end{multline}
  by the relations in the non-torsion part:
  \begin{itemize}
  \item $p\cdot \mu_p = u^{p-2}\sigma u$,
  \item $p\cdot v_0^n\mu_{p^{n+1}} = u^{p^{n+1}-p^n}v_0^{n-1}\mu_{p^n}$ for any $n\geq 1$,
  \end{itemize}
  and the relations in the torsion part:
  \begin{itemize}
  \item $v_0^{h} \sigma u \mu_{ap^n} = 0$ for any $a\geq 1$ not divisible by $p$, $n \geq 1$, and $h \geq n$,
  \item $u^{p^{n-h}-2}\cdot v_0^h\sigma u \mu_{ap^n} = 0$ for any $a \geq 1$ not divisible by $p$,  $n\geq 1$, and $0 \leq h \leq n - 1$,
  \item $p\cdot \sigma u \mu_{(bp+p-1)p^n} = v_0\sigma u\mu_{(bp+p-1)p^n}+u^{p^{n+1}-p^n}v_0^{\nu(b)}\sigma u \mu_{bp^{n+1}}$ for any $b\geq 1$ and any $n\geq 1$,
  \item $p\cdot v_0^h\sigma u \mu_{ap^n} = v_0^{h+1}\sigma u \mu_{ap^n}$ for any $a\geq 1$ not divisible by $p$, $n\geq 1$ and $h\geq 1$ or $h=0$ not in the previous case.
  \end{itemize}
  The degrees are:
  \begin{equation*}
    \begin{aligned}
      & |\mu_{kp}| = 2kp - 1 \\
      & |\sigma u| = 3 \\
      & |v_0| = 0 \\
      & |u| = 2
    \end{aligned}
  \end{equation*}
  and $\nu$ is the $p$-adic valuation.
\end{theorem}

As studied in \cite{angeltveit2010topological} for $\THH_*(\ell)$, the torsion modules of $\THH_*(ku)$ are divided into periodic submodules $T_n$ for $n\geq 1$. Each $T_n$ corresponds to the submodules of the torsion elements of degrees between $|\sigma u \mu_{p^n}|= 2p^n+2$ and $|\sigma u\mu_{2p^n}| - 1 = 2(2p^n) +1$. Each of these appears $p-1$ times, by replacing the leftmost class with $\sigma u\mu_{kp^n}$ for $1 \leq k \leq p - 1$, and $p$ copies (as submodules or quotients) of $T_n$ are present in $T_{n+1}$, so $T_n$ appears an infinite number of times. In the following figures, the generators are named and placed on the bottom horizontal line; the rest of the non-zero classes are indicated by a $\circ$ when they come from $\THH_*(\ell)$, a $\bullet$ otherwise; going straight up indicates a multiplication by $p$, and going upward and right is a multiplication by $u$; when two lines go up from a single class, it means the multiplication by $p$ is the sum of the two elements reached. None of the named classes comes from $\THH_*(\ell)$.

\input{fig1}

We can see that
\begin{equation*}
  \THH_*(ku) \neq ku_* \otimes_{\ell_*} \THH_*(\ell).
\end{equation*}
Since $ku$ is free and finitely generated over $\ell$, $ku$ is flat over $\ell$, so that $\THH(ku)$ and $ku \wedge_\ell \THH(\ell)$ are not weakly equivalent.

The extension of scalars, however, does yield an injection, and in fact a short exact sequence
\begin{equation*}
  \begin{tikzcd}
    0 \rar & ku_*\otimes_{\ell_*} \THH_*(\ell) \rar & \THH_*(ku) \rar & \mathcal{C} \rar & 0
  \end{tikzcd}
\end{equation*}
where the cokernel $\mathcal{C}$ can be presented as the quotient of the $\Z_{(p)}[u]$-module
\begin{equation*}
  P_{p-2}(u)\otimes \Z_{(p)}\{1,\,\sigma u,\,\sigma u\mu_{ap^n},\, n\geq 1,\, a\geq 1,\, \mbox{$a$ not divisible by $p$}\}
\end{equation*}
by the relation $p^{n} \sigma u \mu_{ap^n} = 0$ for any $a$ and $n$, $a$ not divisible by $p$.

\appendix
\section{Appendix: table of the spectral sequences used}
\centerline{
  \begin{tabular}{|c|c|c|c|}\hline
    Name & Type & First page & Target \\\hline
    \multirow{2}*{\eqref{ss:lZ}} &  \multirow{2}*{Brun} & $E^2_{n,m} = \THH_n(H\Z_{(p)}; H(H\Z_{(p)}\wedge_\ell H\Z_{(p)})_m)$ &   \multirow{2}*{$\THH_*(\ell;H\Z_{(p)})$} \\
         & & $\cong \THH_n(H\Z_{(p)})\botimes E(\sigma v_1)_m$ & \\ \hline
    \eqref{ss:l} & Bockstein &  $E^1_{n,m} = \THH_n(\ell;H\Z_{(p)})\botimes P(v_1)_m$ &  $\THH_*(\ell)$ \\ \hline
    \multirow{2}*{\eqref{ss:uT}} & Truncated & $E^1_{n,m} = \THH_n(ku;H\Z_{(p)})\botimes P_{p-1}(u)_m$ &   \multirow{2}*{$\THH_*(ku;ku/v_1)$} \\
         & Bockstein & $\cong (\THH_*(H\Z_{(p)})\otimes E(\sigma u))_n \botimes P_{p-1}(u)_m$ & \\ \hline
    \multirow{3}*{\eqref{ss:uTB}} &  \multirow{3}*{Brun} & $E^2_{n,m} = \THH_n(ku/v_1;H(ku/v_1\wedge_{ku}ku/v_1)_m)$ &    \multirow{3}*{$\THH_*(ku;ku/v_1)$} \\
         & & $\cong (\THH_*(H\Z_{(p)}) \otimes E(\sigma u)\otimes \Gamma(\varphi u))_n$ & \\
         & & $\botimes (E(\sigma v_1) \otimes P_{p-1}(u))_m$ & \\ \hline
    \eqref{ss:v1} & Bockstein & $E^1_{n,m} = \THH_n(ku;ku/v_1)\botimes P(v_1)_m$ & $\THH_*(ku)$ \\ \hline
    \eqref{ss:u} & Bockstein & $E^1_{n,m} = \THH_n(ku;H\Z_{(p)})\botimes P(u)_m$ & $\THH_*(ku)$ \\ \hline
  \end{tabular}}
\begin{equation*}
  \begin{aligned}
    & |\mu_{kp}| = (2kp-1,0) \mbox{, $k\geq 1$ the generators of $\THH_*(H\Z_{(p)})$} \\
    & |\sigma v_1| = (0,2p-1)\\
    & |v_1| = (0,2(p-1)) \\
    & |\sigma u| = (3,0) \\
    & |\lambda_1| = (2p-1,0) \\
    & |\mu_1| = (2p,0) \\
    & |u| = (0,2) \\
    & |\varphi u| = (2p,0)
  \end{aligned}
\end{equation*}
On the left side of the $\botimes$, the generators have bidegrees lying on the horizontal axis; on the right, on the vertical axis.

\bibliographystyle{plain}
\bibliography{bib}

\end{document}